\documentclass[11pt,a4paper,reqno]{amsart}
\usepackage{amsmath}
\usepackage{amssymb}
\usepackage{euscript}
\usepackage[all]{xy}
\usepackage{color}
\usepackage{setspace}
\usepackage[colorlinks=true,linktocpage=true,pagebackref=false, citecolor=black,linkcolor=black]{hyperref}

\usepackage{tikz,pgfplots}
\pgfplotsset{compat=1.10}
\usepgfplotslibrary{fillbetween}
\usetikzlibrary{cd,arrows,decorations.pathmorphing,backgrounds,automata,positioning,fit,matrix}
\pgfplotsset{soldot/.style={color=black,only marks,mark=*}} \pgfplotsset{holdot/.style={color=black,fill=white,only marks,mark=*}}

\newtheorem{thm}{Theorem}[section]

\newtheorem{lem}[thm]{Lemma}
\newtheorem{prop}[thm]{Proposition}

\newtheorem{cor}[thm]{Corollary}

\theoremstyle{definition}
\newtheorem{defn}[thm]{Definition}

\newtheorem{defns}[thm]{Definitions}

\theoremstyle{remark}
\newtheorem{remark}[thm]{Remark}
\newtheorem{remarks}[thm]{Remarks}

\newtheorem{examples}[thm]{Examples}

 \newcommand{\N}{{\mathbb N}}
\newcommand{\Z}{{\mathbb Z}} \newcommand{\R}{{\mathbb R}}
 \newcommand{\C}{{\mathbb C}}

\newcommand{\gtp}{{\mathfrak p}} \newcommand{\gtq}{{\mathfrak q}}
\newcommand{\gtm}{{\mathfrak m}} \newcommand{\gtn}{{\mathfrak n}}
\newcommand{\gta}{{\mathfrak a}} \newcommand{\gtb}{{\mathfrak b}}
\newcommand{\gtP}{{\mathfrak P}} \newcommand{\gtQ}{{\mathfrak Q}}

\newcommand{\Ddd}{{\EuScript D}^{\diam}}
\newcommand{\Zzd}{{\EuScript Z}^{\diam}}
\newcommand{\Dd}{{\EuScript D}}
\newcommand{\Zz}{{\EuScript Z}}
\newcommand{\Dn}{{D}}
\newcommand{\Zn}{{Z}}

\newcommand{\Pp}{{\EuScript P}}

\newcommand{\reg}{\operatorname{reg}}
\newcommand{\Int}{\operatorname{Int}}
\newcommand{\dist}{\operatorname{dist}}

\newcommand{\Spec}{\operatorname{Spec}}
\newcommand{\ord}{\operatorname{ord}}
\newcommand{\cl}{\operatorname{Cl}}
\newcommand{\diam}{{\text{\tiny$\displaystyle\diamond$}}}
\newcommand{\Specs}{\operatorname{Spec}}
\newcommand{\Speca}{\operatorname{Spec^*}}
\newcommand{\Specd}{\operatorname{Spec^{\diam}}}
\newcommand{\betas}{\operatorname{\beta}}
\newcommand{\betaa}{\operatorname{\beta^*}}
\newcommand{\betad}{\operatorname{\beta^\diam}}

\newcommand{\x}{{\tt x}}  
 \renewcommand{\t}{{\tt t}}

\newcommand{\veps}{\varepsilon}

\newcommand{\ol }{\overline}

\newcommand{\Cc}{\mathcal{C}}

\makeatletter \def\l@subsection{\@tocline{2}{0pt}{1pc}{5pc}{}} \def\l@subsection{\@tocline{2}{0pt}{2pc}{6pc}{}} \makeatother
\begin{document}

\title[Semialgebraic branched coverings]{Spectral maps associated to\\ semialgebraic branched coverings}
\author{E. Baro}
\author{Jose F. Fernando}
\author{J.M. Gamboa}
\address{Departamento de \'Algebra, Geometr\'ia y Topolog\'ia, Facultad de Ciencias Matem\'aticas, Universidad Complutense de Madrid, 28040 MADRID (SPAIN)}
\email{eliasbaro@pdi.ucm.es, josefer@mat.ucm.es, jmgamboa@mat.ucm.es}

\date{30/08/2020}
\thanks{Authors supported by Spanish STRANO MTM2017-82105-P and Grupos UCM 910444}

\begin{abstract}
In this article we prove that a semialgebraic map $\pi:M\to N$ is a branched covering if and only if its associated spectral map is a branched covering. In addition, such spectral map has a neat behavior with respect to the branching locus, the ramification set and the ramification index. A crucial fact to prove the preceding result is the characterization of the prime ideals whose fibers under the previous spectral map are singletons. 
\end{abstract}

\maketitle
%\setcounter{tocdepth}{2}
%{\small
%\begin{spacing}{0.01}
%\tableofcontents
%\end{spacing}}

\section{Introduction}\label{s3}
The primary goal of semialgebraic geometry is to study the set of solutions of a finite system of polynomial inequalities in a finite number of variables with coefficients in the field $\R$ of real numbers or, more generally, in an arbitrary real closed field. Frequently, one wants to do this without using polynomial data, as it happens in classical (complex) Algebraic Geometry, where one often avoids working explicitly with the systems of polynomials equalities and non-equalities involved. After the pioneer work of Delfs-Knebusch \cite{dk2}, where they introduced locally semialgebraic spaces and locally semialgebraic maps between them, real algebraic geometers realized the need of constructing their abstract counterpart. 

A subset $M\subset\R^m$ is a \textit{basic semialgebraic} if it can be written as
$$
M:=\{x\in\R^m:\ f(x)=0,g_1(x)>0,\ldots,g_{\ell}(x)>0\}
$$
for some polynomials $f,g_1,\ldots,g_{\ell}\in\R[\x_1,\ldots,\x_m]$. The finite unions of basic semialgebraic sets are called {\em semialgebraic sets}. A continuous map $f:M\to N$ between semialgebraic sets $M\subset\R^m$ and $N\subset\R^n$ is {\em semialgebraic} if its graph is a semialgebraic subset of $\R^{m+n}$. In general, semialgebraic map refers to a (non necessarily continuous) map whose graph is semialgebraic. However, as most of semialgebraic functions and maps appearing in this work are continuous, we omit for the sake of readability the continuity condition when referring to them. By Tarski-Seidenberg's Theorem semialgebraic sets can be characterized as the first order definable sets in the pure field structure of $\mathbb{R}$ (see \cite{dries}).

The sum and product of functions defined pointwise endow the set ${\mathcal S}(M)$ of semialgebraic functions on $M$ with a natural structure of (commutative) $\R$-algebra with unit. The subset ${\mathcal S}^*(M)$ of bounded semialgebraic functions on $M$ is an $\R$-subalgebra of ${\mathcal S}(M)$. We write ${\mathcal S}^\diam(M)$ to refer indistinctly to both rings and we denote $\Specd(M):=\Spec({\mathcal S}^\diam(M))$ the Zariski spectra of ${\mathcal S}^\diam(M)$ endowed with the Zariski topology. Recall that $M$ is a dense subset of $\Specd(M)$. We denote $\betad(M)$ the maximal spectrum of ${\mathcal S}^\diam(M)$, that is, the set of closed points of $\Specd(M)$. As ${\mathcal S}^\diam(M)$ is a Gelfand ring (that is, each prime ideal of ${\mathcal S}^\diam(M)$ is contained in a unique maximal ideal of ${\mathcal S}^\diam(M)$, see \cite{cc}), there exists a natural retraction ${\tt r}_M:\Specd(M)\to\betad(M)$, which is continuous \cite{mo}. Gelfand-Kolmogorov Theorem implies that $\betas(M)$ and $\betaa(M)$ are homeomorphic \cite[Thm.10.1]{t} (see \cite[Thm.3.5]{fg1} for an alternative proof). 

Each semialgebraic map $\pi:M\to N$ has associated a homomorphism of $\R$-algebras $\varphi_{\pi}^\diam:{\mathcal S}^\diam(N)\to{\mathcal S}^\diam(M),\ g\mapsto g\circ\pi$. Thus, one has morphisms
\begin{align*}
&\Specd(\pi):\Specd(M)\to\Specd(N),\ \gtp\mapsto(\varphi_{\pi}^\diam)^{-1}(\gtp),\\
&\betad(\pi):={\tt r}_N\circ\Specd(\pi)|_{\betad(M)}:\betad(M)\to\Specd(N)\to\betad(N),
\end{align*}
which are continuous and `extend' $\pi:M\to N$.

Morphisms between algebraic varieties over algebraically closed fields induce homomorphisms between their coordinate rings and the latter induce morphisms between their Zariski spectra. This is the classical approach to study morphisms between `geometric varieties' via `abstract morphisms' between affine schemes. In the real setting it was not clear neither which are the right rings of functions to deal with nor which should be the `real affine schemes'. The pioneer works \cite{br} of Brumfiel and \cite{cc} of Carral--Coste pointed out that (continuous) semialgebraic functions provide a suitable setting. 

Rings of semialgebraic functions present a key property: their Zariski and real spectra are canonically homeomorphic \cite[III.\S1]{s5}. The theory of the real spectrum introduced by Coste and Roy \cite[\S 7]{bcr} provides powerful tools to understand the interplay between the geometric and abstract settings. It is worthwhile to mention the theory of real closed spaces and real closed rings developed by Schwartz \cite{s2,s3,s4,s5,s6,s7}, which is much more than the abstraction of the geometric locally semialgebraic spaces studied in \cite{dk2} by Delfs--Knebusch. They are powerful tools, which involve a deep knowledge of Commutative Algebra, that help to understand the real spectrum of a ring $A$, as they reduce its study to decipher the Zariski spectrum of its \em real closure \em (which is a real closed ring universally associated to $A$, see \cite{s5}). For instance, the real closure of the ring of polynomials $\R[\x_1,\ldots,\x_n]$ is the ring ${\mathcal S}(\R^n)$ of (continuous) semialgebraic functions on $\R^n$.

The literature quoted above has a foundational nature. The articles \cite{bfg,fe1,fe2,fe3,fg1,fg2,fg3,fg4,fg5,fg6} are devoted to understand the relationship between a semialgebraic map $\pi:M\to N$ and its spectral map $\Specd(\pi):\Specd(M)\to\Specd(N)$ (using basic techniques that involve less prerequisites than the techniques commented above). This article focuses on the preceding relationship when $\pi:M\to N$ is a semialgebraic branched covering.

Branched coverings constitute a common and useful tool in many areas of Mathematics that appears often in Knot Theory, Orbifolds (quotients of manifolds under the discontinuous action of a group), (complex) Algebraic Geometry, (complex) Analytic Geometry, Riemann surfaces, etc. Given two topological spaces $X$ and $Y$, a continuous map $\pi:X\to Y$ is a \em finite quasi-covering \em if it is a separated, open, closed, surjective map whose fibers are finite (\S\ref{quasi}). Inspired by the theory of complex analytic coverings, we propose a notion of \em branched covering \em (\S\ref{branched}) adapted to find a consistent definition of ramification index function. Roughly speaking, a finite quasi-covering $\pi:X\to Y$ is a \em branched covering \em if $\pi$ locally behaves as a covering with a constant number of sheets (after removing certain subset with dense complementary, called the \em ramification set \em ${\mathcal R}_{\pi}$). The ramification set ${\mathcal R}_{\pi}$ is the image under $\pi$ of the \em branching locus \em ${\mathcal B}_{\pi}$, which is the set of points of $X$ at which $\pi$ is not a local homeomorphism. If the fibers at the points of $Y\setminus{\mathcal R}_{\pi}$ have constant cardinality $d$, we say that $\pi$ is a \em $d$-branched covering\em. There is a well-defined notion of \em ramification index \em at a point $x\in M$. Intuitively, it is the number of sheets that $\pi$ has close to $x$. Example \ref{notbranched} is an enlightening (counter-)example that shows the subtleties of the definition of branched covering.

A first goal is to analyze the notions above in the semialgebraic setting. Semialgebraic maps with similar properties have been already studied. For instance, Schwartz characterized openness of semialgebraic maps $\pi:M\to N$ with finite fibers \cite[Thm.13]{s3}: \em a semialgebraic map $\pi$ with finite fibers is open if and only if the homomorphism $\varphi_{\pi}:{\mathcal S}(N)\to{\mathcal S}(M)$ is flat\em. 
 
The main result of this article analyzes the behavior of the spectral map associated to a semialgebraic branched covering. 

\begin{thm}[Spectral map associated to a semialgebraic branched covering]\label{bc} 
Let $M\subset\R^m$ and $N\subset\R^n$ be semialgebraic sets and let $\pi:M\to N$ be a semialgebraic map. The following assertions are equivalent:
\begin{itemize}
\item[(i)] $\pi$ is a branched covering. 
\item[(ii)] $\Specd(\pi):\Specd(M)\to\Specd(N)$ is a branched covering.
\item[(iii)] $\betad(\pi):\betad(M)\to\betad(N)$ is a branched covering. 
\end{itemize}
In addition, if such is the case, then:
\begin{itemize}
\item ${\mathcal B}_{\Specd(\pi)}=\cl_{\Specd(M)}({\mathcal B}_\pi)$ and ${\mathcal B}_{\betad(\pi)}=\cl_{\betad(M)}({\mathcal B}_\pi)$.
\item ${\mathcal R}_{\Specd(\pi)}=\cl_{\Specd(N)}({\mathcal R}_\pi)$ and ${\mathcal R}_{\betad(\pi)}=\cl_{\betad(N)}({\mathcal R}_\pi)$.
\end{itemize}
\end{thm}

The openness, closedness and surjectivity of $\Specd(\pi)$ follow from \cite{fg3,s2}. As it is well-known, the topological space $\Specd(M)$ is in general not Hausdorff. Thus, it is not clear why the spectral map $\Specd(\pi)$ of a semialgebraic branched covering $\pi:M\to N$ should be a separated map. We will prove this in Proposition \ref{sep2} through an analysis of the effect over ${\mathcal S}^\diam(M)$ of symmetric polynomials via the semialgebraic branched covering $\pi:M\to N$ (similarly to \cite[Thm. 12, Ch. III]{gr} for complex analytic coverings). Once this is proved (and consequently $\Specd(\pi)$ is a finite quasi-covering), it remains to be shown that the spectral map $\Specd(\pi)$ is in fact a branched covering. To ease the presentation we include in Section \ref{s5} several (surely known) technical results of topological nature that make the proof of Theorem \ref{bc} more readable.

Another important tool is the study of the set of points of $X$ at which `there is a complete collapse of the fibers of a finite quasi-covering $\pi:X\to Y$'. More precisely, the \em collapsing set \em $\Cc_\pi$ of a finite quasi-covering $\pi:X\to Y$ is the set of points $x\in X$ such that the fiber $\pi^{-1}(\pi(x))$ is a singleton. Given a semialgebraic $d$-branched covering $\pi:M\to N$, our purpose is to characterize the collapsing set $\Cc_{\Specd(\pi)}$ of the spectral map $\Specd(\pi):\Specd(M)\to\Specd(N)$. To that end we introduce in Section \ref{mu} a map $\mu^\diam:{\mathcal S}^\diam(M)\to{\mathcal S}^\diam(N)$, where $\mu^\diam(f)(y)$ is `intuitively' the weighted arithmetic mean with respect to the ramification index of the values of $f\in{\mathcal S}^\diam(M)$ at the points of the finite fiber $\pi^{-1}(y)$. The homomorphism $\varphi^\diam_{\pi}$ endows ${\mathcal S}^\diam(M)$ with a natural structure of ${\mathcal S}^\diam(N)$-module and the map $\mu^\diam$ is a homomorphism of ${\mathcal S}^\diam(N)$-modules. 

\begin{thm}\label{colapse} 
Let $\pi:M\to N$ be a semialgebraic $d$-branched covering. Then 
\begin{itemize}
\item[(i)] $\Cc_{\Specd(\pi)}$ is the set of prime ideals of ${\mathcal S}^\diam(M)$ that contain $\ker(\mu^\diam)$. 
\item[(ii)] $\Cc_{\Specd(\pi)}=\cl_{\Specd(M)}(\Cc_\pi)$.
\item[(iii)] $\Cc_{\betad(\pi)}$ is the set of maximal ideals of ${\mathcal S}^\diam(M)$ that contain $\ker(\mu^\diam)$. 
\item[(iv)] $\Cc_{\betad(\pi)}=\cl_{\betad(M)}(\Cc_\pi)$.
\end{itemize}
\end{thm}

Semialgebraic branched coverings were implicitly introduced by Brumfiel in \cite{br}. Given a semialgebraic set $M\subset\R^m$ and a closed equivalence relation $E\subset M\times M$ such that the projection $\Pi:E\to M$ is proper, it was proved in \cite[Thm. 1.4]{br} the existence of a semialgebraic set $N\subset\R^m$, a surjective semialgebraic map $\pi:M\to N$ and a homeomorphism $g:M/E\to N$ such that $\pi=g\circ\rho$, where $\rho:M\to M/E$ is the natural projection. Scheiderer studied more general quotients in \cite{sch} and again semialgebraic branched coverings appear. We present an enlightening related example at the end of this paper (the Bezoutian covering). Another source of examples are algebraic morphisms between complex algebraic curves \cite{bcg1, bcg2} (after considering their real intrinsic structures).

\noindent{\bf Differences between the ${\mathcal S}$ and ${\mathcal S}^*$ cases.} 
In this article we have tried to unify the proofs of the results for the rings ${\mathcal S}(M)$ and ${\mathcal S}^*(M)$ whenever it has been possible. For some results certain differences appear because of the particularities of both settings. In the $\mathcal{S}$-case we have the machinery of $z$-ideals, which is not available in the ${\mathcal S}^*$-case (see Subsection \ref{subsecAddition} and Corollary \ref{minz}). For instance, the proof of Theorem \ref{colapse} is done in the $\mathcal{S}$-case using $z$-ideals and we show how one reduces the proof of Theorem \ref{colapse} in the $\mathcal{S}^*$-case to the preceding one.

\section{Branched coverings}\label{s5}
We begin this section with some general topological facts. For each subset $A$ of a topological space $X$ we denote $\cl_X(A)$ and $\Int_X(A)$ the closure and the interior of $A$ in $X$. In addition, $\#(A)$ denotes the cardinality of $A$. The following results are straightforward, but very useful for our discussion below.

\begin{lem}\label{trivial}
Let $\pi:X\to Y$ be a surjective map and let $Z\subset Y$. Denote $T:=\pi^{-1}(Z)$. Then for each set $A\subset X$, we have $\pi(A\cap T)=\pi(A)\cap\pi(T)$. In addition,
\begin{itemize}
\item[(i)] If $\pi$ is open, $\pi|_T:T\to Z$ is open.
\item[(ii)] If $\pi$ is closed, $\pi|_T:T\to Z$ is closed.
\end{itemize}
\end{lem}

\begin{lem}\label{opcl}
Let $\pi:X\to Y$ be a continuous map and let $A\subset X$ and $B\subset Y$. Then
\begin{itemize}
\item[(i)] If $\pi$ is open, $\cl_X(\pi^{-1}(B))=\pi^{-1}(\cl_Y(B))$.
\item[(ii)] If $\pi$ is closed, $\pi(\cl_X(A))=\cl_Y(\pi(A))$.
\end{itemize}
\end{lem}

In this article we study certain classes of coverings in two different topological contexts: the semialgebraic and the spectral one. In order to analyze both contexts simultaneously, we work in a general enough topological setting. This has forced us to be careful when writing down the proofs. Of course, many of the them can be shortened and simplified if we add additional hypothesis to the topological spaces (for instance, if they are subspaces of affine spaces or if the spaces are locally connected).

\subsection{Finite quasi-coverings}\label{quasi}
A continuous map $\pi:X\to Y$ is \em separated \em if each pair of points in the same fiber admits disjoint neighborhoods in $X$. A \em finite quasi-covering \em is a separated, open and closed surjective map $\pi:X\to Y$ whose fibers are finite. 

\begin{remark}\label{fqc}
Let $\pi:X\to Y$ be a finite quasi-covering and let $Z\subset Y$. Denote $T:=\pi^{-1}(Z)$. Then $\pi|_{T}:T\to Z$ is a finite quasi-covering by Lemma \ref{trivial}.
\end{remark}

We define next some special neighborhoods related to the points of the spaces that appear in a finite quasi-covering.

\begin{lem}[Characteristic and distinguished neighborhood]\label{disting} 
Let $\pi:X\to Y$ be a finite quasi-covering and let $y\in Y$ be such that its fiber $\pi^{-1}(y):=\{x_1,\dots,x_r\}$ has $r$ distinct points. Let $W^{x_1},\dots,W^{x_r}\subset X$ be pairwise disjoint open neighborhoods of $x_1,\ldots,x_r$. Then there exists an open neighborhood $V_0\subset Y$ of $y$ such that for each open neighborhood $V\subset V_0$ of $y$ there exist pairwise disjoint open neighborhoods $U^{x_1},\dots,U^{x_r}\subset X$ of the points $x_1,\dots,x_r$ satisfying $U^{x_i}\subset W^{x_i}$,
\begin{equation}\label{cubierta}
\pi^{-1}(V)=\bigsqcup_{j=1}^rU^{x_j}\quad\text{and}\quad\pi(U^{x_j})=V\quad\text{for}\quad 1\leq j\leq r. 
\end{equation}
In addition, $\#(\pi^{-1}(z))\geq r$ for each $z\in V$. \em We say that $V$ is a \em distinguished neighborhood of $y$ \em (with respect to $\pi$) and $U^{x_1},\ldots, U^{x_r}$ is a family of \em characteristic neighborhoods \em (with respect to $V$). We say that $U$ is a \em characteristic neighborhood \em of $x\in X$ if $U$ is a member of a family of characteristic neighborhoods with respect to a distinguished open neighborhood of $\pi(x)$.
\end{lem}
\begin{proof}
The existence of pairwise disjoint open neighborhoods $W^{x_1},\dots,W^{x_r}\subset X$ of the points $x_1,\ldots,x_r$ is guaranteed because $\pi$ is a separated map. As $\pi$ is an open and closed map,
$$
V_0:=\Big(Y\setminus\pi\Big(X\setminus\bigcup_{j=1}^rW^{x_j}\Big)\Big)\cap\bigcap_{j=1}^r\pi(W^{x_j})\subset Y
$$
is an open neighborhood of $y$ and $\pi^{-1}(V_0)\subset\bigcup_{j=1}^rW^{x_j}$. Let $V\subset V_0$ be an open neighborhood of $y$. Define $U^{x_j}:=W^{x_j}\cap\pi^{-1}(V)\subset W^{x_j}\subset X$, which is an open neighborhood of $x_j$. Then $U^{x_i}\cap U^{x_j}=\varnothing$ if $i\neq j$ and the reader can check that equalities \eqref{cubierta} follow. Once this is checked, the last part of the statement is clear.
\end{proof}

\begin{remark}\label{intersection}
Let $\pi:X\to Y$ be a finite quasi-covering and let $y\in Y$. Let $\pi^{-1}(y):=\{x_1,\ldots,x_r\}$ and let $W_1,\ldots,W_k$ be open neighborhoods of $x_1,\ldots,x_k$ for some $k\leq r$. Suppose that $V$ is a distinguished neighborhood of $y$ and let $U_1,\ldots,U_r$ be characteristic neighborhoods of $x_1,\dots,x_r$ with respect to $V$. Then there exist a distinguished open neighborhood $\widetilde{V}\subset V$ of $y$ and characteristic neighborhoods $\widetilde{U}_i$ with respect to $\widetilde{V}$ (for $i=1,\ldots,r$) satisfying:
\begin{itemize}
\item $\widetilde{U}_i=\pi^{-1}(\widetilde{V})\cap U_i\cap W_i\subset U_i\cap W_i$ for $i=1,\ldots,k$
\item $\widetilde{U}_i=\pi^{-1}(\widetilde{V})\cap U_i\subset U_i$ for $i=k+1,\ldots,r$.
\end{itemize}

Indeed, it is enough to apply Lemma \ref{disting} to the family $U_1\cap W_1,\ldots, U_k\cap W_k,U_{k+1},\ldots, U_r$. Note that for each $z\in \widetilde{V}$, we have $\#(\pi^{-1}(z)\cap \widetilde{U}_i)=\#(\pi^{-1}(z)\cap U_i)$.\qed
\end{remark}

The reader can check that if a distinguished neighborhood $V$ is connected, then the family of characteristic neighborhoods with respect to $V$ coincides with the family of connected components of the inverse image of $V$ under the finite quasi-covering.

\begin{cor}\label{ccs0}
Let $\pi:X\to Y$ be a finite quasi-covering and let $y\in Y$. Write $\pi^{-1}(y)=\{x_1,\ldots,x_r\}$. Let $V$ be a distinguished neighborhood of $y$ and let $U^{x_1},\ldots,U^{x_r}$ be a family of characteristic neighborhoods with respect to $V$. Assume that $V$ is connected. Then $U^{x_1},\ldots,U^{x_r}$ are the connected components of $\pi^{-1}(V)$.
\end{cor}

\begin{defn}
If $\pi:X\to Y$ is a finite quasi-covering, the \em branching locus \em of $\pi$ is the closed set $\mathcal{B}_\pi$ of all points belonging to $X$ at which $\pi$ is not a local homeomorphism. The \em ramification set \em of $\pi$ is the closed set ${\mathcal R}_{\pi}:=\pi({\mathcal B}_{\pi})\subset Y$. The \em regular locus \em of $\pi$ is the open set 
$$
X_{\reg}:=X\setminus\pi^{-1}({\mathcal R}_{\pi})\subset X.
$$
We say that $\pi:X\to Y$ is a \em $d$-unbranched covering \em (for some integer $d\geq1$) if $X_{\reg}=X$ and the cardinality of each fiber is equal to $d$. In case we do not want to specify the integer $d$, we will say that $\pi$ is an \em unbranched covering\em. It is important to keep in mind that the fibers of an unbranched covering have constant cardinality (see Examples \ref{notbranched}).
\end{defn}

\begin{lem}\label{genbranch} 
Let $\pi:X\to Y$ be a finite quasi-covering. Then $y\in Y\setminus{\mathcal R}_{\pi}$ if and only if there exists an open neighborhood $W\subset Y$ of $y$ such that the cardinality of the fiber $\pi^{-1}(z)$ for each $z\in W$ is constant.
\end{lem}
\begin{proof}Let $y\in Y$ and denote $\pi^{-1}(y):=\{x_1,\ldots,x_r\}$. If $y\in Y\setminus{\mathcal R}_{\pi}$ there exists an open neighborhood $U_i$ of $x_i$ such that $\pi|_{U_i}:U_i\to \pi(U_i)$ is a homeomorphism for each $i=1,\ldots,r$. By Remark \ref{intersection} we can assume that $U_1,\ldots,U_r$ is a family of characteristic neighborhoods with respect to their common image $W$. Thus, the cardinality of $\pi^{-1}(z)$ is constant for $z\in W$.

Conversely, if there exists a neighborhood $W\subset Y$ of $y$ such that the cardinality of each fiber $\pi^{-1}(z)$ for $z\in W$ is constant, after shrinking $W$ if necessary we may assume by Remark \ref{intersection} that there exists a family of characteristic neighborhoods $U^{x_1},\ldots,U^{x_r}$ with respect to $W$. Now, each restriction $\pi|_{U^{x_i}}:U^{x_i}\to W$ must be injective. By Remark \ref{fqc} the surjective map $\pi|_{U^{x_i}}$ is open and closed and therefore it is a homeomorphism, as required. 
\end{proof}

\begin{cor}\label{max}
Let $\pi:X\to Y$ be a finite quasi-covering and suppose that $d:=\sup\{\#(\pi^{-1}(y)):\ y\in Y\}<+\infty$. Let $y\in Y$ be such that $\#(\pi^{-1}(y))=d$. Then $y\in Y\setminus{\mathcal R}_{\pi}$. 
\end{cor}
\begin{proof}
Denote $\pi^{-1}(y):=\{x_1,\ldots,x_d\}$. Let $V$ be a distinguished neighborhood of $y$ with respect to $\pi$. By Lemma \ref{disting} $d=\#(\pi^{-1}(y))\leq\#(\pi^{-1}(z))\leq d$ for each $z\in V$, so the cardinality of the fiber $\pi^{-1}(z)$ for each $z\in V$ is constant. By Lemma \ref{genbranch} $y\in Y\setminus{\mathcal R}_{\pi}$, as required.
\end{proof}

We finish this part with a topological property of certain finite quasi-coverings that will be used in the proof of Theorem \ref{bc}.

\begin{lem}\label{nowhere}
Let $\pi:X\to Y$ be a finite quasi-covering such that $X_{\reg}$ is dense in $X$. If $Z$ is a closed nowhere dense subset of $X$ then $\pi(Z)$ is a closed nowhere dense subset of $Y$. 
\end{lem}
\begin{proof}
As $\mathcal{R}_\pi$ is a closed nowhere dense subset of $Y$, if $\text{int}_Y(\pi(Z)\setminus \mathcal{R}_\pi)=\varnothing$, then $\text{int}_Y(\pi(Z))=\varnothing$. Thus, we can assume $X_{\reg}=X$ and the statement follows straightforwardly because $\pi:X_{\reg}\to Y\setminus\mathcal{R}_\pi$ is a local homeomorphism with finite fibers.
\end{proof}

\subsection{Branched coverings}\label{branched}
Let $\pi:X\to Y$ be a finite quasi-covering, let $y\in Y$ and pick $x\in\pi^{-1}(y)$. Let $V$ be an open neighborhood of $y$ and let $U$ be an open and closed subset of $\pi^{-1}(V)$. By Lemma \ref{trivial} the restriction map $\pi|_U:U\to \pi(U)$ is also a finite quasi-covering and $\mathcal{B}_{\pi|_U}=\mathcal{B}_{\pi}\cap U$, so that $\mathcal{R}_{\pi|_U}\subset \mathcal{R}_{\pi}$ and $X_{\text{reg}}\cap U\subset U_{\text{reg}}$.

The next definition of branched covering is proposed to avoid the pathology described in Examples \ref{notbranched} below. Such type of examples can never appear in the context of complex analytic coverings \cite[Ch. III]{gr}. The similarity of branched coverings to complex analytic coverings allows us to define properly the ramification index and to use symmetric polynomials when dealing with branched coverings, see \S\ref{sympol}. These tools are crucial to prove Theorem \ref{colapse}.

\begin{defns}
With the notations introduced above:

(i) A characteristic neighborhood $U$ of $x$ with respect to a distinguished neighborhood $V$ of $y$ such that the restriction $\pi|_{U_{\reg}}:U_{\reg}\to V\setminus{\mathcal R}_{\pi|_U}$ is an unbranched covering is called an \em exceptional neighborhood of $x$ \em (with respect to $V$). If each member of a family of characteristic neighborhoods with respect to $V$ is exceptional, then $V$ is a \em special neighborhood of $y$\em. In that case, we say that such a family is a \em family of exceptional neighborhoods with respect to $V$\em. When it is clear from the context we will omit with respect to $V$.

The number $b_\pi(x)$ of sheets of $\pi|_{U_{\reg}}$ (the common cardinality of its fibers) is the \em ramification index \em of $x$ relative to $U$. We will show in Lemma \ref{indexwell} that $b_\pi(x)$ does not depend on $U$.

(ii) We say that $\pi$ is a \em branched covering \em if $X_{\reg}$ is a dense subset of $X$ and each $y\in Y$ admits a special neighborhood. 

(iii) Let $\pi:X\to Y$ be a branched covering. For each $y\in Y\setminus{\mathcal R}_\pi$ there exists an open neighborhood $W$ of $y$ such that the cardinality of the fibers at the points of $W$ is constant (Lemma \ref{genbranch}). We say that $\pi$ is a \em $d$-branched covering \em if this constant is the same positive integer $d$ for each point $y\in Y\setminus{\mathcal R}_\pi$. 
\end{defns}

\begin{lem}\label{cuenta0}
Let $\pi:X\to Y$ be a finite quasi-covering. Let $y\in Y$ and $x\in\pi^{-1}(y)$. Let $U\subset X$ be a characteristic neighborhood of $x$ with respect to a distinguished neighborhood $V\subset Y$ of $y$. Let $G$ be an open dense subset of $U$ such that the cardinality of the fibers of the restriction $\pi|_{G}:G\to \pi(G)$ is constant and equal to $d\in\N$. Then $\pi|_{U_{\reg}}:U_{\reg}\to V\setminus{\mathcal R}_{\pi|_U}$ is a $d$-unbranched covering and $U$ is an exceptional neighborhood of $x$ with respect to $V$. 
\end{lem}
\begin{proof}
The branching locus of the finite quasi-covering $\pi|_{U_{\reg}}:U_{\reg}\to V\setminus{\mathcal R}_{\pi|_U}$ is empty. We claim: \em $\#(\pi^{-1}(z)\cap U)=d$ for each $z\in\pi(G)$\em.

Let $z\in\pi(G)$. By hypothesis $\#(\pi^{-1}(z)\cap G)=d$. Assume $m:=\#(\pi^{-1}(z)\cap U)>d$ and let $V_0\subset\pi(G)$ be a distinguished neighborhood of $z$. Let $U_{01},\ldots,U_{0m}$ be a family of characteristic neighborhoods with respect to $V_0$. As $G$ is dense in $U$, the intersection $G\cap U_{0i}$ is dense in $U_{0i}$, so by Lemmas \ref{trivial} and \ref{opcl} $\bigcap_{i=1}^m\pi(G\cap U_{0i})$ is a dense open subset of $V_0$. If $z'\in\bigcap_{i=1}^m\pi(G\cap U_{0i})$, then $d=\#(\pi^{-1}(z')\cap G)\geq m>d$, which is a contradiction.

We prove next: \em $\#(\pi^{-1}(y)\cap U)=d$ for each $y\in V\setminus{\mathcal R}_{\pi|_U}$\em.

Let $y\in V\setminus{\mathcal R}_{\pi|_U}$. By Lemma \ref{genbranch} there exists an open neighborhood $W\subset V$ such that the cardinality of the fiber $\pi^{-1}(z)\cap U$ for each $z\in W$ is a constant $c$. By Remark \ref{intersection} we can assume that $\pi^{-1}(W)\cap U_{\reg}=\bigsqcup^c_{i=1} U_i$ where each $\pi|_{U_i}$ is a homeomorphism onto $W$. As $G$ is an open dense subset of $U$, we deduce that $G\cap U_i$ is an open dense subset of $U_i$ for each $i=1,\ldots,c$. Thus, $\bigcap^c_{i=1} \pi(G\cap U_i)$ is a dense open subset of $W$. If $z\in \bigcap^c_{i=1}\pi(G\cap U_i)$, then the fiber $\pi^{-1}(z)\cap U$ has $d$ elements, so $c=d$. Thus, $\pi|_{U_{\reg}}:U_{\reg}\to V\setminus{\mathcal R}_{\pi|_U}$ is a $d$-unbranched covering, as required.
\end{proof}

As a straightforward application of Remark \ref{intersection} we have the following.

\begin{remark}\label{intersection:branched}
Let $\pi:X\to Y$ be a finite quasi-covering, let $y\in Y$ and denote $\pi^{-1}(y):=\{x_1,\ldots,x_r\}$. If each $x_i$ has an exceptional neighborhood $U_i$ for $i=1,\ldots,r$ and $U_1,\ldots,U_r$ are pairwise disjoint, then $y$ has a special neighborhood `as small as needed'. We are not assuming that $U_1,\ldots,U_r$ is a family of characteristic neighborhoods with respect to their common image, but that each $U_i$ belongs to a possibly different family of characteristic neighborhoods. 
\end{remark}

\begin{lem}[Ramification index]\label{indexwell}
Let $\pi:X\to Y$ be a branched covering and let $x\in X$. The ramification index $b_\pi(x)$ of $x$ does not depend on the chosen exceptional neighborhood of $x$. 
\end{lem}
\begin{proof}
We may assume that $\pi$ is a $d$-branched covering. Let $y\in Y$ and $\pi^{-1}(y):=\{x_1,\ldots,x_r\}$. Let $V$ be a special open neighborhood of $y$ and let $U_1,\ldots,U_r$ be a family of exceptional neighborhoods of $x_1,\ldots,x_r$ with respect to $V$. Let $V'$ be another special open neighborhood of $y$ and let $U'_1,\ldots,U'_r$ be a family of exceptional neighborhoods of $x_1,\ldots,x_r$ with respect to $V'$. We denote by $b^{U_i}_{\pi}(x_i)$ the ramification index of $x_i$ with respect to $U_i$, that is, $b^{U_i}_{\pi}(x_i)=\#(\pi^{-1}(z)\cap U_i)$ for each $z\in V\setminus \mathcal{R}_\pi$. Next, we consider the ramification index $b^{U'_i}_{\pi}(x_i)$ of $x_i$ with respect to $U'_i$. After reordering the indexes, it is enough to show: $b^{U_1}_{\pi}(x_1)=b^{U'_1}_{\pi}(x_1)$ (that is, the case $i=1$).

By Remark \ref{intersection} there exists a distinguished open neighborhood $\widetilde{V}\subset V$ of $y$ and characteristic neighborhoods $\widetilde{U}_1,\ldots\widetilde{U}_r$ with respect to $\widetilde{V}$ such that $\widetilde{U}_1=\pi^{-1}(\widetilde{V})\cap U_1\cap U'_1$ and $\widetilde{U}_i=\pi^{-1}(\widetilde{V})\cap U_i$ for $i=2,\ldots,r$. In particular, for each $z\in \widetilde{V}\setminus \mathcal{R}_{\pi}$ we have $b^{U_i}_{\pi}(x_i)=\# (\pi^{-1}(z)\cap U_i)=\# (\pi^{-1}(z)\cap \widetilde{U_i})$ for $i=2,\ldots,r$ and 
\begin{align*}
b^{U'_1}_{\pi}(x_1)=\ &\# (\pi^{-1}(z)\cap U'_1)\geq\# (\pi^{-1}(z)\cap U_1\cap U'_1)=\# (\pi^{-1}(z)\cap\widetilde{U}_1)=\\
=\ &d-\sum^r_{i=2}\# (\pi^{-1}(z)\cap\widetilde{U}_i)=d-\sum^r_{i=2}b^{U_i}_{\pi}(x_i)=b^{U_1}_{\pi}(x_1). 
\end{align*}
We have proved $b^{U_1}_{\pi}(x_1)\leq b^{U'_1}_{\pi}(x_1)$. The same argument shows $b^{U'_1}_{\pi}(x_1)\leq b^{U_1}_{\pi}(x_1)$, so $b^{U_1}_{\pi}(x_1)= b^{U'_1}_{\pi}(x_1)$, as required.
\end{proof}

\begin{remarks}\label{cuenta}
Let $\pi:X\to Y$ be a branched covering, let $x\in X$ and let $y\in Y$. 

(i) Let $V$ be a special neighborhood of $y$. Then the restriction map $\pi|_{\pi^{-1}(V)}:\pi^{-1}(V)\to V$ is a $d_y$-branched covering where $d_y:=\sum_{x\in\pi^{-1}(y)}b_\pi(x)$. Thus, $d_y=\#(\pi^{-1}(w))$ for each $w\in V\setminus\mathcal{R}_{\pi|_{\pi^{-1}(V)}}$ and $d_z=d_y$ for each $z\in V$.

(ii) If $Y$ is connected then $\pi$ is a $d$-branched covering for some $d\geq1$.

By Remark \ref{cuenta}(i) the set $Y_d:=\{y\in Y:d_y=d\}$ is open for each $d\in \mathbb{N}$. Consequently, each set $Y_d$ is also closed. As $Y$ is connected, we deduce that $Y=Y_d$ for some $d\in \mathbb{N}$.

(iii) Write $\pi^{-1}(y):=\{x_1,\ldots,x_r\}$. Let $W^{x_i}\subset X$ be an open neighborhood of $x_i$ for each $i=1,\ldots,r$. Then there exists an open special neighborhood $V\subset Y$ of $y$ and a family of exceptional neighborhoods $U^{x_1},\ldots,U^{x_r}$ with respect to $V$ such that $U^{x_i}\subset W^{x_i}$ for $i=1,\ldots,r$.

Let $V_0\subset Y$ be a special neighborhood of $y$ and let $U_0^{x_i}$ be the corresponding exceptional neighborhood of $x_i$. By Remark \ref{intersection} there exists a distinguished neighborhood $V\subset V_0$ of $y$ and a family of characteristic neighborhoods $U^{x_i}\subset U_0^{x_i}\cap W^{x_i}$ of $x_i$ with respect to $V$ such that $\pi^{-1}(V)\cap U_0^{x_i}\cap W^{x_i}=U^{x_i}$ for each $i=1,\ldots,r$. Using Lemma \ref{trivial} and Remark \ref{intersection} the reader can check: \em $V$ is a special neighborhood of $y$ and each $U^{x_i}$ is an exceptional neighborhood of $x_i$\em. 

(iv) Let $U$ be an exceptional neighborhood of $x$. Then $b_{\pi}(x):=\max\{\#(\pi^{-1}(y)\cap U):\ y\in \pi(U)\}$. In addition, for each $x'\in U$ we have $b_{\pi}(x')\leq b_{\pi}(x)$. Thus, for each $e\in \mathbb{N}$ the set $\{b_{\pi}\leq e\}$ is open in $M$, so $b_{\pi}$ is upper semi-continuous.

By definition $b_{\pi}(x)=\#(\pi^{-1}(y))$ for each $y\in\pi(U)\setminus\mathcal{R}_{\pi|_U}$. Pick $y'\in V:=\pi(U)$ and denote $\pi^{-1}(y')\cap U=\{x'_1,\ldots,x'_\ell\}$. By Remark \ref{cuenta}(iii) there exist a special neighborhood $V'\subset V$ of $y'$ and exceptional neighborhoods $U'_1,\ldots,U'_\ell\subset U$ of $x'_1,\ldots,x'_\ell$. Each intersection $U_{\reg}\cap U'_i$ is a dense open subset of $U'_{i}$, so $\bigcap^\ell_{i=1} \pi(U_{\reg}\cap U'_i)$ is a dense open subset of $V'$. Thus, $\ell\leq b_{\pi}(x)$.

Next, given $x'\in U$ there exists by Remark \ref{cuenta}(iii) an exceptional neighborhood $U'$ of $x'$ contained in $U$. It follows that 
$$
b_{\pi}(x')=\max\{\#(\pi^{-1}(y)\cap U'):\ y\in \pi(U')\}\leq \max\{\#(\pi^{-1}(y)\cap U):\ y\in \pi(U)\}=b_{\pi}(x).
$$

(v) $b_{\pi}(x)=1$ if and only if $x\notin \mathcal{B}_{\pi}$.

The `only if part' follows from Remark \ref{cuenta}(iv), whereas the `if part' follows from Remark \ref{intersection} and Remark \ref{cuenta}(iii).\qed
\end{remarks}

\subsubsection{Behavior of branched coverings under restriction}
We analyze next how branched coverings behave under restriction.

\begin{lem}\label{restr}
Let $\pi:X\to Y$ be a branched covering. Let $W\subset Y$ be an open set. Let $W\subset Z\subset\cl(W)$ and denote $T:=\pi^{-1}(Z)$. Then $\pi|_{T}:T\to Z$ is a branched covering, ${\mathcal B}_{\pi|_T}={\mathcal B}_\pi\cap T$, $T_{\reg}=X_{\reg}\cap T$ and ${\mathcal R}_{\pi|_T}={\mathcal R}_\pi\cap Z$.
\end{lem}
\begin{proof}
We use freely in this proof the following: \em If $U$ is an open subset of $X$ and $A$ is a dense subset of $U$, then $A\cap T$ is dense in $U\cap T$\em. 

By Remark \ref{fqc} $\pi|_{T}$ is a finite quasi-covering. Observe that ${\mathcal B}_{\pi|_T}\subset{\mathcal B}_\pi\cap T$, so ${\mathcal R}_{\pi|_T}\subset{\mathcal R}_\pi\cap Z$ and $X_{\reg}\cap T\subset T_{\reg}$. As $X_{\reg}$ is dense in $X$, it follows $T_{\reg}$ is dense in $T$.

Pick a point $y\in Z$ and write $\pi^{-1}(y):=\{x_1,\ldots,x_r\}\subset T$. Let $V$ be a special neighborhood of $y$ and let $U_1,\ldots,U_r$ be a family of exceptional neighborhoods of $x_1,\ldots,x_r$ with respect to $V$. Thus, each restriction $\pi|_{U_{i,\reg}}:U_{i,\reg}\to V\setminus{\mathcal R}_{\pi|_{U_i}}$ is a $(b_\pi(x_i))$-unbranched covering. Regarding $\pi|_T:T\to Z$, the open set $V\cap Z$ is a distinguished neighborhood of $y$ and $U_1\cap T,\ldots,U_r\cap T$ is a family of characteristic neighborhoods with respect to $V\cap Z$. In addition,
$$
\pi|_{(U_i\cap T)_{\reg}}:(U_i\cap T)_{\reg}\to(V\cap Z)\setminus{\mathcal R}_{\pi|_{U_i\cap T}}
$$
is also a $(b_\pi(x_i))$-unbranched covering by Lemma \ref{cuenta0} because $U_{i,\reg}\cap T$ is a dense open subset of $U_i\cap T$ and the cardinality of the fibers of $\pi|_{U_{i,\reg}\cap T}$ equals $b_\pi(x_i)$. We conclude that $U_1\cap T,\ldots,U_r\cap T$ is a family of exceptional neighborhoods with respect to the special neighborhood $V\cap T$. In particular, $\pi|_{\pi^{-1}(V)\cap T}:\pi^{-1}(V)\cap T\to V\cap Z$ is a $d$-branched covering where $d:=\sum_{i=1}^rb_\pi(x_i)$.

By Remark \ref{cuenta}(v) $x_i\in{\mathcal B}_{\pi|_T}$ if and only if $b_\pi(x_i)>1$, so ${\mathcal B}_{\pi|_T}={\mathcal B}_\pi\cap T$ and by Lemma \ref{trivial} ${\mathcal R}_{\pi|_T}={\mathcal R}_\pi\cap Z$. Finally,
$$
T_{\reg}=T\setminus \pi^{-1}(\mathcal{R}_{\pi|_T})=T\setminus (\pi^{-1}(\mathcal{R}_{\pi})\cap\pi^{-1}(Z))=T\setminus \pi^{-1}(\mathcal{R}_{\pi})=T\cap X_{\reg},
$$
as required.
\end{proof}

The following result is a straightforward consequence of the previous lemma.

\begin{cor}\label{rcc}
Let $\pi:X\to Y$ be a map and assume that $Y$ has finitely many connected components $Y_1,\ldots,Y_r$. Denote $X_i:=\pi^{-1}(Y_i)$ for each $i=1,\ldots,r$. Then $\pi:X\to Y$ is a branched covering if and only if for each $i=1,\ldots, k$ there exists an integer $d_i\geq 1$ such that ${\pi|_{X_i}:X_i\to Y_i}$ is a $d_i$-branched covering.
\end{cor}

\begin{lem}\label{opcl2}
Let $\pi:X\to Y$ be a branched covering. Let $T\subset X$ be an open and closed set and denote $Z:=\pi(T)$. Then $\pi|_T:T\to Z$ is a branched covering, ${\mathcal B}_{\pi|_T}={\mathcal B}_\pi\cap T$, ${\mathcal R}_{\pi|_T}\subset{\mathcal R}_\pi\cap Z$ and $X_{\reg}\cap T\subset T_{\reg}$.
\end{lem}
\begin{proof}
First, observe that $Z$ is open and closed in $Y$. As $T$ is an open and closed subset of $\pi^{-1}(Z)$, the restriction map $\pi|_T:T\to Z$ is a finite quasi-covering with ${\mathcal B}_{\pi|_T}={\mathcal B}_\pi\cap T$, ${\mathcal R}_{\pi|_T}\subset{\mathcal R}_\pi\cap Z$ and $T\cap X_{\reg}\subset T_{\reg}$. As $X_{\reg}$ is dense in $X$ and $T$ is open, $X_{\reg}\cap T$ is dense in $T$, so $T_{\reg}$ is dense in $T$.

Pick $y\in Z$ and write $\pi^{-1}(y):=\{x_1,\ldots,x_r\}$. We may assume $\pi^{-1}(y)\cap T=\{x_1,\ldots,x_s\}$ for some $s\leq r$. Let $V\subset Y$ be a special neighborhood of $y$ and $U_1,\ldots,U_r$ be a family of exceptional neighborhoods of $x_1,\ldots,x_r$ with respect to $V$. As each set $U_1\cap T,\ldots, U_s\cap T$ is open, we can assume by Remark \ref{cuenta}(iii) that $U_1,\ldots,U_s\subset T$, so $V\subset Z$. As $\pi|_{U_{i,\reg}}:U_{i,\reg}\to V\setminus{\mathcal R}_{\pi|_{U_i}}$ is an unbranched covering for $i=1,\ldots,s$, we conclude that $V$ is a special neighborhood with respect to $\pi|_T:T\to Z$ and $U_1,\ldots,U_s$ are exceptional neighborhoods. Thus, $\pi|_T$ is a branched covering, as required.
\end{proof}

\subsubsection{Some special branched coverings}
We propose next some mild sufficient conditions under which we can guarantee that a finite quasi-covering is a branched covering.

\begin{lem}\label{sc}
Let $\pi:X\to Y$ be a finite quasi-covering such that $d:=\sup\{\#(\pi^{-1}(y)):\ y\in Y\}<+\infty$ and for each $y\in Y$ there exists a distinguished neighborhood $V$ of $y$ such that $V\cap(Y\setminus{\mathcal R}_\pi)$ is connected. Then $\pi$ is a branched covering. 
\end{lem}
\begin{proof}
We show first: \em $X_{\reg}=X\setminus\pi^{-1}({\mathcal R}_{\pi})$ is dense in $X$\em, or equivalently: \em $Y\setminus{\mathcal R}_{\pi}$ is dense in $Y$\em. 

Suppose that ${\mathcal R}_{\pi}$ contains a non-empty open subset $V\subset Y$. Define $V_k:=\{y\in V: \#(\pi^{-1}(y))=k\}$ and note that $V=V_1\sqcup\cdots\sqcup V_d$. In particular, for some $1\leq k\leq d$ the interior $\Int(V_k)\cap V$ of $V_k$ in $V$ is not empty. By Lemma \ref{genbranch} we get $\Int(V_k)\cap V\subset Y\setminus{\mathcal R}_\pi$, which is a contradiction.

Fix $y\in Y$. By hypothesis there exists a distinguished neighborhood $V$ of $y$ such that ${V\cap(Y\setminus{\mathcal R}_\pi)}$ is connected. As $Y\setminus{\mathcal R}_\pi$ is dense in $Y$, the intersection $V\cap(Y\setminus{\mathcal R}_\pi)$ is dense in $V$, so $V$ is connected. Write $\pi^{-1}(y):=\{x_1,\ldots,x_r\}$.

Let $U_1,\ldots,U_r$ be a family of characteristic neighborhoods of $x_1,\ldots,x_r$ with respect to $V$. Then $\pi|_{U_i}:U_i\to V$ is a finite quasi-covering and ${\mathcal R}_{\pi|_{U_i}}\subset{\mathcal R}_{\pi}\cap V$ for $i=1,\ldots,r$. Thus, each difference $V\setminus {\mathcal R}_{\pi|_{U_i}}$ is connected. By Lemma \ref{genbranch} the cardinality of the fibers of the restriction map $\pi|_{U_{i,\reg}}:U_{i,\reg}\to V\setminus {\mathcal R}_{\pi|_{U_i}}$ is locally constant. As $V\setminus {\mathcal R}_{\pi|_{U_i}}$ is connected, $\pi|_{U_{i,\reg}}$ is a $k_i$-unbranched covering for some $k_i\in\N$, so $U_i$ is an exceptional neighborhood of $x_i$, as required.
\end{proof}

\begin{remark}
The previous situation is quite common. It arises for instance when one considers the underlying real structure of a complex irreducible analytic germ and analyzes local parameterization theorem \cite[Ch.2.B \& Ch.3.B]{gr} (as a consequence of Noether's normalization lemma \cite[Ch.5.Ex.16]{am}). Analogously, it appears when consider the underlying real structure of Noether's normalization lemma of a complex irreducible algebraic set \cite[Ch.5.Ex.16]{am}. 
\end{remark}

\subsection{Collapsing set of a finite quasi-covering}\label{s4}
Our next purpose is to analyze the set of points at which there exists a complete collapse of the fibers of a finite quasi-covering. 

\begin{defn}
Let $\pi:X\to Y$ be a finite quasi-covering. We define the \em collapsing set of $\pi$ \em as $\Cc_\pi:=\{x\in X:\pi^{-1}(\pi(x))=\{x\}\}$.
\end{defn}
\begin{remarks} 
(i) The collapsing set of a finite quasi-covering is a closed set.

Let $\pi:X\to Y$ be a finite quasi-covering. By Lemma \ref{disting} the cardinality of the fibers at the points close to a given point $y\in Y$ is greater than or equal to the cardinality of $\pi^{-1}(y)$. Thus, the set $S$ of points whose fiber contains at least two points is an open subset of $Y$ and ${\mathcal C}_\pi=\pi^{-1}(X\setminus S)$ is a closed subset of $X$. 

(ii) If $\pi:X\to Y$ is a $d$-branched covering, its collapsing set is $\Cc_{\pi}=\{x\in X:\ b_\pi(x)=d\}$.\qed
\end{remarks}

We will need also the following result.
\begin{lem}\label{colapseinB}
Let $\pi:X\to Y$ be a finite quasi-covering. Let $C\subset X$ and $D\subset Y$ be such that $\pi|_C:C\to D$ is a $d$-branched covering for some $d>1$. Then $\mathcal{C}_{\pi|_C}\subset \mathcal{B}_\pi$.
\end{lem}
\begin{proof}
Let $x\in \mathcal{C}_{\pi|_C}$ and suppose that $x\notin \mathcal{B}_\pi$. Then there exists an open neighborhood $W$ of $x$ such that $\pi|_W:W\to \pi(W)$ is a homeomorphism. As $\pi|_C:C\to D$ is an open continuous map, $\pi|_{W\cap C}:W\cap C\to\pi(W\cap C)$ is an open continuous bijective map, so it is a homeomorphism and $\pi(W\cap C)$ is an open subset of $D$. Thus, by Remark \ref{cuenta}(v) we deduce $b_{\pi|_C}(x)=1$. As $x\in \mathcal{C}_{\pi|_C}$, we conclude $1=b_{\pi|_C}(x)=d$, which is a contradiction.
\end{proof}

\subsection{Semialgebraic branched coverings} 
As one can expect a \em semialgebraic finite quasi-cover\-ing \em is a finite quasi-covering that is in addition a semialgebraic map. As semialgebraic sets are Hausdorff spaces, we deduce the following from Lemma \ref{disting}.

\begin{cor}\label{semiquasi}
Let $\pi:M\to N$ be an open, closed, surjective semialgebraic map with finite fibers between semialgebraic sets $M$ and $N$. Then $\pi$ is a finite quasi-covering and each point $y\in N$ admits a basis of distinguished semialgebraic neighborhoods with respect to $\pi$. 
\end{cor}

Concerning the branching locus and ramification set we have the following.

\begin{lem}\label{ss} 
Let $\pi:M\to N$ be a semialgebraic finite quasi-covering.
\begin{itemize}
\item[(i)] The cardinality of the fibers of $\pi$ is bounded by a common constant.
\item[(ii)] $\mathcal{B}_\pi,M_{\reg},\mathcal{C}_{\pi}\subset M$ and ${\mathcal R}_{\pi}\subset N$ are semialgebraic sets.
\end{itemize}
\end{lem}
\begin{proof}
(i) This follows from cell decomposition of semialgebraic sets \cite[Corollary 3.7]{dries}. 

(ii) A point $x\in\mathcal{B}_\pi$ if and only if the restriction $\pi|_U:U\to\pi(U)$ is not a homeomorphism for each open semialgebraic neighborhood $U\subset M$ of $x$. As $\pi$ is open, continuous and surjective, we deduce $x\in\mathcal{B}_\pi$ if and only if the restriction $\pi|_U$ is not injective on each open semialgebraic neighborhood $U\subset M$ of $x$. Consequently, ${\mathcal B}_{\pi}$ is a semialgebraic subset of $M$, so ${\mathcal R}_{\pi}=\pi({\mathcal B}_{\pi})$ is a semialgebraic subset of $N$ and $M_{\reg}:=M\setminus\pi^{-1}({\mathcal R}_{\pi})$ is a semialgebraic subset of $M$. The set $\mathcal{C}_{\pi}=\{x\in M:\ \pi^{-1}(x)=\{x\}\}$ is also semialgebraic, as required.
\end{proof}

A \em semialgebraic branched covering \em is a map $\pi:M\to N$ that is simultaneously a branched covering and a semialgebraic map. Lemma \ref{indexwell}, Corollary \ref{rcc} and Lemma \ref{sc} apply readily in the semialgebraic case. In order to show some subtleties hidden in the definition of branched covering we provide next an example of a semialgebraic finite quasi-covering that is not a semialgebraic branched covering. 

\begin{examples}\label{notbranched} 
(i) Consider the semialgebraic subsets of $\R^2$ defined by
\begin{equation*}
\begin{split}
M_1:&=([-2,0]\times\{3/2\})\cup\Big\{(x,y)\in\R^2:\ 0\leq x\leq 2,\ y=\frac{3\pm x}{2}\Big\},\\
M_2:&=([0,2]\times\{3\})\cup\Big\{(x,y)\in\R^2:\ -2\leq x\leq 0,\ y=\frac{6\pm x}{2}\Big\}
\end{split}
\end{equation*}
and $N:=[-2,2]\times\{0\}$. The projection $\pi:M:=M_1\cup M_2\to N,\ (x,y)\mapsto x$ is a semialgebraic finite quasi-covering, but it is not a semialgebraic branched covering. The branching locus of $\pi$ is $\mathcal{B}_\pi:=\{p_1:=(0,3/2),p_2:=(0,3)\}$, so the ramification set is ${\mathcal R}_{\pi}:=\pi({\mathcal B}_{\pi})=\{q:=(0,0)\}$. We have $\#(\pi^{-1}(q))=2$ and $\#(\pi^{-1}(y))=3$ for each point $y\in N\setminus\{q\}$.

\begin{center}
\begin{figure}[ht]
\begin{tikzpicture}[scale=1]
\draw (2,5.3) node{$M_2$};
\draw (2,1.7) node{$M_1$};
\draw (2,0.3) node{$N$};
\begin{axis}[unit vector ratio=1.25 1.25 1.25,axis x line=middle,
	axis y line=middle]	
\addplot[domain=0:2,blue,line width=1.5pt,smooth]{(3+x)/2};
\addplot[domain=0:2,blue,line width=1.5pt,smooth]{(3-x)/2};
\addplot[domain=0:2,blue,line width=1.5pt,smooth]{3};
\addplot[domain=-2:0,blue,line width=1.5pt,smooth]{(6+x)/2};
\addplot[domain=-2:0,blue,line width=1.5pt,smooth]{(6-x)/2};
\addplot[domain=-2:0,blue,line width=1.5pt,smooth]{3/2};
\addplot[domain=-2:2,red,line width=3pt,smooth]{0};
\addplot[soldot] coordinates{(0,1.5)(0,3)(0,0)};
\end{axis}
\end{tikzpicture}
\caption{Projection $\pi:M:=M_1\cup M_2\to N$}
\end{figure}
\end{center}

The regular locus of $\pi$ is the dense subset $M_{\reg}=M\setminus\{{p_1,p_2}\}$ of $M$. Suppose that $\pi$ is a semialgebraic branched covering. Then there exists a distinguished open semialgebraic neighborhood $V$ of $q$ in $N$ and open semialgebraic neighborhoods $U_i$ of $p_i$ in $M$ such that the restriction $\pi|_{M_{\reg}\cap U_i}:M_{\reg}\cap U_i\to (N\setminus{\mathcal R}_{\pi})\cap V$ is a semialgebraic unbranched covering for $i=1,2$. This is false because
\begin{itemize}
\item the cardinality of the fibers of $\pi|_{M_{\reg}\cap U_1}$ at the points of $(N\setminus{\mathcal R}_{\pi})\cap V\cap\{x<0\}$ is $1$, 
\item the cardinality of the fibers of $\pi|_{M_{\reg}\cap U_1}$ at the points of $(N\setminus{\mathcal R}_{\pi})\cap V\cap\{x>0\}$ is $2$. 
\end{itemize}

(ii) A similar pathology appears in the (more restrictive) real algebraic case if one considers
$$
X:=\{x=(y-2)^2\}\cup\{x=-(y+2)^2\}\cup\{y=2\}\cup\{y=-2\},\quad Y:=\{y=0\}
$$
and $\pi:X\to Y,\ (x,y)\mapsto x$. The previous map is a finite quasi-covering, the general fiber has $4$ points, but it is not a branched covering. 
\end{examples}

\begin{prop}[Semialgebraic ramification index]
Let $\pi:M\to N$ be a semialgebraic branched covering. The ramification index function $b_\pi:M\to\N\subset\R$ has semialgebraic graph.
\end{prop}
\begin{proof}
Let $d:=\max\{\#(\pi^{-1}(y)):\ y\in N\}<+\infty$. Observe that $b_\pi(M)\subset\{1,\ldots,d\}$. For each $k=1,\ldots,d$ define ${\mathcal B}_{\pi,k}=\{x\in M:\ b_\pi(x)=k\}$, ${\mathcal B}_{\pi,k}^*=\{x\in M:\ b_\pi(x)\geq k\}$ and ${\mathcal B}_{\pi,d+1}^*=\varnothing$. As ${\mathcal B}_{\pi,k}={\mathcal B}_{\pi,k}^*\setminus{\mathcal B}_{\pi,k+1}^*$ for $k=1,\ldots,d$, to prove that $b_\pi$ is a semialgebraic map, it is enough to check: \em ${\mathcal B}_{\pi,k}^*$ is a semialgebraic set for each $k=1,\ldots,d$\em.

It holds
\begin{multline*}
{\mathcal B}_{\pi,k}^*=\{x\in M:\ \forall\veps>0,\ \exists u_1,\ldots,u_k\in M,\ \|x-u_i\|<\veps\\ 
(\forall i=1,\ldots,k),\ u_i\neq u_j,\ \pi(u_1)=\cdots=\pi(u_k)\},
\end{multline*}
so it is described by a first order formula and it is a semialgebraic set, as required.
\end{proof}

\section{Branched coverings and spectral maps}
In this section we analyze the properties of the spectral maps associated to semialgebraic finite quasi-coverings and branched coverings. These results will be applied in Sections \ref{mu} and \ref{s6}. One of the main results of this section is Proposition \ref{sep2}, where we prove that if $\pi:M\to N$ is a semialgebraic branched covering, the spectral map $\Specd(\pi)$ is a finite quasi-covering.

\subsection{Preliminaries on rings of semialgebraic functions}\label{prelim} 
Let $M\subset\R^m$ be a semialgebraic set and let $f\in{\mathcal S}(M)$. We denote $\Zn(f):=\{x\in M:\ f(x)=0\}$ and $\Dn(f):=M\setminus\Zn(f)$. If $N$ is a closed semialgebraic subset of $M$, the semialgebraic function $g:=\dist(\cdot,N)\in{\mathcal S}(M)$ satisfies $Z(g)=N$. In fact, after replacing $g$ by $\frac{g}{1+g^2}$ we may assume in addition that $g$ is bounded. The restriction homomorphism ${\mathcal S}^\diam(M)\to{\mathcal S}^\diam(N),\ f\mapsto f|_N$ is by \cite[Thm.3]{dk1} surjective. Denote
\begin{equation}
\begin{split}
\Zzd(f):&=\{\gtp\in\Specd(M):\ f\in\gtp\},\\
\Ddd(f):&=\Specd(M)\setminus\Zzd(f)=\{\gtp\in\Specd(M):\ f\notin\gtp\}.
\end{split} 
\end{equation}
The semialgebraic set $M$ is identified with a dense subspace of $\Specd(M)$ via the injective map ${\tt j}_M:M\hookrightarrow\Specd(M),\ x\mapsto\gtm^\diam_x$, where $\gtm^\diam_x:=\{f\in{\mathcal S}^\diam(M):\ f(x)=0\}$ is the maximal ideal of ${\mathcal S}^\diam(M)$ associated to $x$. In particular, $\Zn(f)=M\cap\Zzd(f)$ and $\Dn(f)=M\cap\Ddd(f)$. Thus, ${\tt j}_M$ is an embedding.

We denote $\betad(M)\subset\Specd(M)$ the set of maximal ideals of ${\mathcal S}^\diam(M)$. As $\gtm_x\in\betad(M)$ for each point $x\in M$, we have $M\subset\betad(M)$. Denote $\partial M:=\betad(M)\setminus M$ and recall that $\betad(M)$ is a compact Hausdorff space \cite[Prop.7.1.25]{bcr}. Schwartz proved in \cite[III.\S1]{s5} that ${\mathcal S}^\diam(M)$ is a real closed ring. Thus, its Zariski spectrum is homeomorphic to its real spectrum (Coste--Roy \cite{cr}) via the support map $\Spec_r({\mathcal S}^\diam(M))\to\Specd(M),\alpha\mapsto\alpha\cap(-\alpha)$. The set of specializations of a point in a real spectrum constitutes a chain \cite[Prop.7.1.23]{bcr}, so the same holds for the set of prime ideals in ${\mathcal S}^\diam(M)$ that contains a given prime ideal. Consequently, ${\mathcal S}^\diam(M)$ is a Gelfand ring \cite{cc}, that is, each prime ideal $\gtp\in\Specd(M)$ is contained in a unique maximal ideal $\gtm\in\betad(M)$. This provides a natural retraction ${\tt r}_M:\Specd(M)\to\betad(M)$, which is continuous \cite[Thm.1.2]{mo}. 

If $\pi:M\to N$ is a semialgebraic map, the induced maps 
\begin{align*}
&\Specd(\pi):\Specd(M)\to\Specd(N),\ \gtp\mapsto\varphi_{\pi}^{-1}(\gtp),\\
&\betad(\pi):={\tt r}_N\circ\Specd(\pi)|_{\betad(M)}:\betad(M)\to\Specd(N)\to\betad(N)
\end{align*} 
are continuous, $\Specd(\pi)|_{M}=\pi$ and $\betad(\pi)|_{M}=\pi$. We recall here the following result. 

\begin{lem}\label{propern} 
Let $\pi:M\to N$ be a semialgebraic map.
\begin{itemize} 
\item[(i)] If $\pi$ is open, closed and surjective, then $\Specd(\pi)$ is open, closed and surjective.
\item[(ii)] Suppose that $\pi$ is surjective. Then it is proper if and only if $\Specd(\pi)(\partial M)=\partial N$.
\end{itemize} 
\end{lem}
\begin{proof}
Statement (i) was proved in \cite[Ch.6]{s2}, whereas (ii) was shown in \cite[Rem.4.2]{fg6}.
\end{proof}

Observe that ${\mathcal W}_M:=\{f\in{\mathcal S}^*(M):\ Z(f)=\varnothing\}$ is a multiplicatively closed subset of ${\mathcal S}^*(M)$ and ${\mathcal S}(M)={\mathcal W}_M^{-1}{\mathcal S}^*(M)$. This is so because each $f\in{\mathcal S}(M)$ can be written as $f=g/h$, where
$$
g:=\frac{f}{1+f^2}\in{\mathcal S}^*(M)\quad\text{and}\quad h:=\frac{1}{1+f^2}\in{\mathcal W}_M.
$$
As ${\mathcal S}(M)={\mathcal W}_M^{-1}{\mathcal S}^*(M)$, there exists a bijection, which is in fact a homeomorphism,
\begin{align*}
\phi&:\Specs(M)\to{\mathfrak S}(M),\ \gtq\mapsto\gtq\cap{\mathcal S}^*(M),\\
\phi^{-1}&:{\mathfrak S}(M)\to\Specs(M),\ \gtq'\mapsto{\mathcal W}_M^{-1}\gtq'
\end{align*}
that preserves inclusions between prime ideals, where ${\mathfrak S}(M)$ is the set of prime ideals of ${\mathcal S}^*(M)$ that do not meet ${\mathcal W}_M$.

\begin{remark}\label{r:maximal} As it happens with rings of continuous functions \cite[Ch.\,7]{gj}, the comparison of the Zariski spectra of ${\mathcal S}(M)$ and $\mathcal{S}^*(M)$ provides a homeomorphism between their maximal spectra (see \cite[\S 10]{t} for a extended version of the Gelfand-Kolmogorov theorem). Explicitly, this homeomorphism is the map $\beta(M)\mapsto\beta^*(M)$ that maps each maximal ideal $\gtm$ of ${\mathcal S}(M)$ to the unique maximal ideal $\gtm^*$ of $\mathcal{S}^*(M)$ that contains the prime $\gtm\cap\mathcal{S}^*(M)$. 

In addition, by \cite[Prop.5.1]{fe1} we have: \em Let $\gtp\in\Spec^*(M)$ be a prime ideal. If $\gtm^*\in\beta^*(M)$ is the unique maximal ideal of ${\mathcal S}^*(M)$ that contains $\gtp$ and $\gtm\in\beta(M)$ is the unique maximal ideal of ${\mathcal S}(M)$ such that $\gtm\cap\mathcal{S}^*(M)\subset\gtm^*$, then either $\gtp\subset\gtm\cap\mathcal{S}^*(M)$ or $\gtm\cap\mathcal{S}^*(M)\subset\gtp$\em.\qed
\end{remark}

Given a semialgebraic map $\pi:M\to N$ the induced homomorphism $\varphi_{\pi}:{\mathcal S}(N)\to{\mathcal S}(M)$ maps ${\mathcal S}^*(N)$ to ${\mathcal S}^*(M)$ and we denote $\varphi^*_{\pi}:=\varphi_{\pi}|_{{\mathcal S}^*(N)}:{\mathcal S}^*(N)\to{\mathcal S}^*(M)$ the restricted homomorphism. 

\begin{lem}[Going down]\label{mintomin} 
Let $\pi:M\to N$ be an open, closed and surjective semialgebraic map. Then 
\begin{itemize}
\item[(i)] The induced homomorphism $\varphi^\diam_{\pi}:{\mathcal S}^\diam(N)\to{\mathcal S}^\diam(M)$ satisfies the going-down property.
\item[(ii)] The spectral map $\Specd(\pi):\Specd(M)\to\Specd(N)$ maps minimal prime ideals of ${\mathcal S}^\diam(M)$ to minimal prime ideals of ${\mathcal S}^\diam(N)$.
\end{itemize}
\end{lem}
\begin{proof}
The $\mathcal{S}^*$-case was proved in \cite[3.6]{fg3}. The $\mathcal{S}$-case is proved from the $\mathcal{S}^*$-case using that ${\mathcal S}(N)={\mathcal W}_N^{-1}{\mathcal S}^*(N)$ and ${\mathcal S}(M)={\mathcal W}_M^{-1}{\mathcal S}^*(M)$. Alternatively, it follows from \cite[Prop.1]{s3} and \cite[Ch.10]{am} using only that $\pi$ is open.
\end{proof}

\subsection{Addition of radical and prime ideals}\label{subsecAddition}
We need some results concerning the addition of radical and prime ideals of rings of semialgebraic functions.

\begin{lem}\label{suma} 
Let $\gta$ be a radical ideal of ${\mathcal S}^\diam(M)$ and let $\gtp$ be a prime ideal of ${\mathcal S}^\diam(M)$. Let $f,g\in{\mathcal S}^\diam(M)$. It holds:
\begin{itemize} 
\item[(i)] If $|f|\leq |g|$ and $g\in\gta$, then $f\in\gta$.
\item[(ii)] The sum of two radical ideals of ${\mathcal S}^\diam(M)$ is a radical ideal.
\item[(iii)] The sum $\gta+\gtp$ is either ${\mathcal S}^\diam(M)$ or a prime ideal of ${\mathcal S}^\diam(M)$. 
\end{itemize}
\end{lem} 
\begin{proof} 
(i) By \cite[Prop.3.8]{s5} or by the axiomatization of real closed rings used in \cite[Introduction]{ps} we have $|f|^2\in|g|{\mathcal S}^\diam(M)$. Thus, $f^4\in g^2{\mathcal S}^\diam(M)\subset\gta$, so $f\in\gta$.

(ii) This follows from \cite[Cor.15]{s1} because ${\mathcal S}^\diam(M)$ is a real closed ring. 

(iii) By (ii) $\gtb:=\gta+\gtp$ is a radical ideal. Let $\Pp:=\{\gtq\in\Specd(M):\ \gtb\subset\gtq\}$. Then $\gtb=\bigcap_{\gtq\in\Pp}\gtq$. As $\Pp$ is a chain with respect to the inclusion, $\gtb={\mathcal S}^\diam(M)$ or $\gtb$ is a prime ideal, as required.
\end{proof}
 
Recall that an ideal $\gta$ of the ring ${\mathcal S}(M)$ is a \em $z$-ideal \em if for each pair of functions $f,g\in{\mathcal S}(M)$ such that $f\in\gta$ and $Z(f)\subset Z(g)$, it holds $g\in\gta$. In particular, $z$-ideals are radical ideals.

\begin{lem}\label{min} 
Let $M\subset\R^n$ be a semialgebraic set.
\begin{itemize} 
\item[(i)] Let $\gta_1,\gta_2$ be two $z$-ideals of ${\mathcal S}(M)$. The sum $\gta_1+\gta_2$ is either ${\mathcal S}(M)$ or a $z$-ideal of ${\mathcal S}(M)$. 
\item[(ii)] Let $\gtp_1,\dots,\gtp_k$ be minimal prime ideals of ${\mathcal S}(M)$. The sum $\gtp:=\gtp_1+\cdots+\gtp_k$ is ${\mathcal S}(M)$ or a prime $z$-ideal of ${\mathcal S}(M)$.
\end{itemize} 
\end{lem}
\begin{proof} 
(i) Suppose that $\gta_1+\gta_2\neq{\mathcal S}(M)$ and let $f,g\in{\mathcal S}(M)$ be such that $f\in\gta_1+\gta_2$ and $Z(f)\subset Z(g)$. Then there exist functions $f_i\in\gta_i$ such that $f=f_1+f_2$. As $Z(f_1)\cap Z(f_2)\subset Z(f)\subset Z(g)$, the function
$$
h:N:=Z(f_1)\cup Z(f_2)\to\R,\ x\mapsto
\begin{cases}
0&\text{if $x\in Z(f_1)$},\\
g(x)&\text{if $x\in Z(f_2)$}\\
\end{cases}
$$
is a well-defined semialgebraic function. By \cite{dk1} there exists $H\in{\mathcal S}(M)$ such that $H|_N=h$. We have $Z(f_1)\subset Z(H)$ and $Z(f_2)\subset Z(g-H)$. As each $\gta_i$ is a $z$-ideal, $H\in\gta_1$ and $g-H\in\gta_2$. Thus, $g=H+(g-H)\in\gta_1+\gta_2$.

(ii) We prove the statement by induction on $k$. All minimal prime ideals of ${\mathcal S}(M)$ are $z$-ideals \cite[Cor.4.7]{fe1}. Suppose $k\geq2$ and let $\gtq:=\gtp_1+\cdots+\gtp_{k-1}$. By induction hypothesis either $\gtq={\mathcal S}(M)$ or $\gtq$ is a prime $z$-ideal. By part (i) the sum $\gtp=\gtq+\gtp_k$ is either ${\mathcal S}(M)$ or a $z$-ideal. In the last case $\gtp$ is by Lemma \ref{suma}(iii) a prime ideal, as required.
\end{proof}

\subsection{Symmetric polynomials and semialgebraic $d$-branched coverings}\label{sympol}
Let us analyze the effect over ${\mathcal S}^\diam(M)$ of symmetric polynomials via a semialgebraic $d$-branched covering $\pi:M\to N$. The approach follows the ideas of the reference \cite[Thm. 12, Ch. III]{gr} concerning complex analytic coverings.

\begin{lem}\label{sym}
Let $\pi:M\to N$ be a semialgebraic $d$-branched covering and let $\sigma\in\R[\x_1,\ldots,\x_d]$ be a symmetric polynomial. Let $f\in{\mathcal S}^\diam(M)$ and define
$$
\sigma(f):N\to\R,\ y\mapsto\sigma(f(x_1),\overset{b_\pi(x_1)}{\ldots},f(x_1),\ldots,f(x_r),\overset{b_\pi(x_r)}{\ldots},f(x_r))
$$
if $\pi^{-1}(y)=\{x_1,\ldots,x_r\}$ (recall that $b_\pi(x_1)+\cdots+b_\pi(x_r)=d$). Then $\sigma(f)\in{\mathcal S}^\diam(N)$.
\end{lem}
\begin{proof}
As $\sigma$ is a symmetric polynomial, $\sigma$ is a well-defined function. We prove first: \em $\sigma(f)$ is continuous on $N$\em. 

Pick $y\in Y$ and write $\pi^{-1}(y):=\{x_1,\ldots,x_r\}$. Fix $\veps>0$. As $\sigma:\R^d\to\R$ is continuous at
$$
p:=(p_1,\ldots,p_d):=(f(x_1),\overset{b_\pi(x_1)}{\ldots},f(x_1),\ldots,f(x_r),\overset{b_\pi(x_r)}{\ldots},f(x_r)),
$$
there exists $\delta>0$ such that if $q:=(q_1,\ldots,q_d)\in\R^d$ and $|p_i-q_i|<\delta$ for $i=1,\ldots,d$, then $|\sigma(p)-\sigma(q)|<\veps$. As $f$ is continuous at $x_1,\ldots,x_r$, there exist open neighborhoods $A^{x_i}$ of $x_i$ such that $|f(z_i)-f(x_i)|<\delta$ for each $z_i\in A^{x_i}$ for $i=1,\ldots,r$. Let $V\subset N$ be a special neighborhood of $y$ and let $U^{x_1},\ldots,U^{x_r}$ be a family of exceptional neighborhoods with $U^{x_i}\subset A^{x_i}$ for $i=1,\ldots,r$ (use Remark \ref{cuenta}(iii)). Pick a point $y'\in V$ and write $\pi^{-1}(y'):=\{z_{11},\ldots,z_{1s_1},\ldots,z_{r1},\ldots,z_{rs_r}\}$ where $z_{ij}\in U^{x_i}$ for $j=1,\ldots,s_i$ and $i=1,\ldots,r$. Using that $\pi|_{M_{\reg}}:M_{\reg}\to N\setminus{\mathcal R}_\pi$ is an unbranched covering, the reader can check that $\sum_{j=1}^{r_i}b_\pi(z_{ij})=b_\pi(x_i)$ for $i=1,\ldots,r$ (Remark \ref{cuenta}(iv) can be useful). As $x_i,z_{ij}\in U^{x_i}\subset A^{x_i}$, we have $|f(x_i)-f(z_{ij})|<\delta$ for $j=1,\ldots,s_i$ and $i=1,\ldots,r$. As $\sum_{j=1}^{r_i}b_\pi(z_{ij})=b_\pi(x_i)$ for $i=1,\ldots,r$, we conclude
\begin{equation*}
\begin{split}
|\sigma(f)(y)-\sigma(f)(y')|&=|\sigma(f(x_1),\overset{b_\pi(x_1)}{\ldots},f(x_1),\ldots,f(x_r),\overset{b_\pi(x_r)}{\ldots},f(x_r))\\
&-\sigma(f(z_{11}),\overset{b_\pi(z_{11})}{\ldots},f(z_{11}),\ldots,f(z_{1s_1}),\overset{b_\pi(z_{1s_1})}{\ldots},f(z_{1s_1}),\\&\ldots,f(z_{r1}),\overset{b_\pi(z_{r1})}{\ldots},f(z_{r1}),\ldots,f(z_{rs_r}),\overset{b_\pi(z_{rs_r})}{\ldots},f(z_{rs_r}))|<\veps.
\end{split}
\end{equation*}
Thus, $\sigma(f)$ is continuous.

In addition, if $f$ is bounded, $\sigma(f)$ is also bounded. We claim: \em $\sigma(f)$ has semialgebraic graph\em.

The restriction $\pi|_{M_{\reg}}:M_{\reg}\to N\setminus{\mathcal R}_\pi$ is a semialgebraic map. For each $y\in N\setminus{\mathcal R}_\pi$ there exist exactly $d$ different points $x_1,\ldots,x_d\in M$ such that $\pi(x_i)=y$ and
$\sigma(f)(y)=\sigma(f(x_1),\ldots,f(x_d))$ (as the polynomial $\sigma$ is symmetric the ordering of the values $f(x_i)$ is not relevant). Thus, the graph of $\sigma(f)|_{N\setminus{\mathcal R}_\pi}$ is a first order definable set, so $\sigma(f)|_{N\setminus{\mathcal R}_\pi}$ is a semialgebraic map. As $\sigma(f)$ is a continuous map, the set $N\setminus{\mathcal R}_\pi$ is dense in $N$ and $\sigma(f)|_{N\setminus{\mathcal R}_\pi}=\sigma(f|_{M\setminus\pi^{-1}({\mathcal R}_\pi}))$ is a semialgebraic function on $N\setminus{\mathcal R}_\pi$, we conclude that the graph $\Gamma(\sigma(f))$ of $\sigma(f)$ is the semialgebraic set $\cl_{M\times\R}(\Gamma(\sigma(f)|_{N\setminus{\mathcal R}_\pi}))$. Thus, $\sigma(f)$ is a semialgebraic function, as required. 
\end{proof}

\subsection{Separated spectral maps}
We prove next a separation result for certain pair of points in $\Specd(M)$, which will allow us to prove in Proposition \ref{sep2} that the spectral map associated to a semialgebraic branched covering is separated.

\begin{lem}[Separation]\label{sep} 
Let $\gtp_1,\gtp_2\in\Specd(M)$ be such that $\gtp_1\not\subset\gtp_2$ and $\gtp_2\not\subset\gtp_1$. Then there exist $f_1,f_2\in{\mathcal S}^\diam(M)$ such that $\gtp_i\in\Ddd(f_i)$ and $f_1f_2=0$. In particular, $\Ddd(f_1)\cap\Ddd(f_2)=\varnothing$.
\end{lem}
\begin{proof} 
As the prime ideals of $\Specd(M)$ that contain a given prime ideal constitute a chain, there exist no prime ideals in $\Specd(M)$ contained in both $\gtp_1$ and $\gtp_2$. Consider the multiplicatively closed subset $\mathfrak{T}:=({\mathcal S}^\diam(M)\setminus \gtp_1)\cdot({\mathcal S}^\diam(M)\setminus\gtp_2)$ of ${\mathcal S}^\diam(M)$. Suppose $0\not\in\mathfrak{T}$. By \cite[Ch.1.Thm.1]{k} there exists a prime ideal $\gtp$ of ${\mathcal S}^\diam(M)$ such that $\gtp\cap\mathfrak{T}=\varnothing$, so $\gtp\subset\gtp_1\cap\gtp_2$, which is a contradiction. Consequently, $0\in\mathfrak{T}$ and the statement holds. 
\end{proof}

The previous type of neighborhoods in $\Specd(M)$ are used below to show (in a very elementary way) the local connectedness of the Zariski spectra of rings of semialgebraic functions.

\begin{lem}\label{neigh}
Let $M\subset\R^m$ be a semialgebraic set and let $f\in{\mathcal S}^\diam(M)$. Denote $D:=D(f)\subset M\subset \betad(M)$. Then $\Ddd(f)\subset\cl_{\Specd(M)}(D)$.
\end{lem}
\begin{proof}
Denote $C:=\cl_M(D)\subset \betad(M)$ and let $\gtp\in\Ddd(f)$. We claim: \em the kernel of the homomorphism $\psi:{\mathcal S}^\diam(M)\to{\mathcal S}^\diam(C),\ g\mapsto g|_C$ is contained in $\gtp$\em. 

Let $h\in{\mathcal S}^\diam(M)$ be such that $h|_C=0$. As $hf=0$ and $\gtp\in\Ddd(f)$, we deduce $h\in\gtp$.

Now, by \cite[Lem.4.3]{fg2} $\gtp\in\cl_{\Specd(M)}(C)$ and we conclude $\Ddd(f)\subset\cl_{\Specd(M)}(C)$, as required.
\end{proof} 

\begin{lem}\label{ccD}
Let $M\subset\R^m$ be a semialgebraic set and let $g\in{\mathcal S}^\diam(M)$. Let $E_1,\ldots,E_s$ be the connected components of $D(g)$. Then the connected components of $\Ddd(g)$ are ${\mathcal V}_i:=\cl_{\Specd(M)}(E_i)\cap\Ddd(g)$ for $i=1,\ldots,s$. In addition, ${\mathcal V}_i=\Ddd(g)\setminus\bigcup_{j\neq i}\cl_{\Specd(M)}(E_j)$ is an open subset of $\Specd(M)$.
\end{lem}
\begin{proof}
It is enough to show: \em each ${\mathcal V}_i$ is connected, $\Ddd(g)=\bigcup_{i=1}^s{\mathcal V}_i$ and ${\mathcal V}_j\cap{\mathcal V}_k=\varnothing$ if $j\neq k$\em. 

As $E_i$ is connected, ${\mathcal V}_i=\cl_{\Specd(M)}(E_i)\cap\Ddd(g)$ is also connected. By Lemma \ref{neigh}
\begin{multline*}
\Ddd(g)=\cl_{\Specd(M)}(\Ddd(g))\cap\Ddd(g)\\
=\cl_{\Specd(M)}\Big(\bigcup_{i=1}^sE_i\Big)\cap\Ddd(g)
=\bigcup_{i=1}^s\cl_{\Specd(M)}(E_i)\cap\Ddd(g)=\bigcup_{i=1}^s{\mathcal V}_i.
\end{multline*}

If $j\neq k$, we have by \cite[Lem.4.5]{fg2}
\begin{equation*}
\begin{split}
{\mathcal V}_j\cap{\mathcal V}_k&=\cl_{\Specd(M)}(E_j)\cap\cl_{\Specd(M)}(E_k)\cap\Ddd(g)\\
&=\cl_{\Specd(M)}(\cl(E_j)\cap\cl(E_k))\cap\Ddd(g)\\
&=\cl_{\Specd(M)}(\cl(E_j)\cap\cl(E_k)\cap D(g))\cap\Ddd(g)=\varnothing,
\end{split}
\end{equation*}
as required.
\end{proof}

\begin{cor}[Local connectedness of Zariski spectra]\label{lczs}
Let $M\subset\R^m$ be a semialgebraic set. Then $\Specd(M)$ is locally connected.
\end{cor}
\begin{proof}
Let $\gtp\in\Specd(M)$ and let ${\mathcal W}\subset\Specd(M)$ be an open neighborhood of $\gtp$. Let $g\in{\mathcal S}^\diam(M)$ be such that $\gtp\in\Ddd(g)\subset{\mathcal W}$. By Lemma \ref{ccD} $\Ddd(g)$ has finitely many connected components ${\mathcal V}_i$, which are open subsets of $\Specd(M)$. We may assume $\gtp\in{\mathcal V}_1$, which is a connected open subset of $\Specd(M)$ contained in ${\mathcal W}$, as required.
\end{proof}

The following result shows (among other things) that spectral maps associated to semialgebraic $d$-branched coverings are separated.

\begin{prop}[Separated spectral map]\label{sep2}
Let $\pi:M\to N$ be a semialgebraic $d$-branched covering. Then 
\begin{itemize}
\item[(i)] $\varphi^\diam_\pi:{\mathcal S}^\diam(N)\to{\mathcal S}^\diam(M)$ is an integral homomorphism.
\item[(ii)] $\Specd(\pi)$ is a finite quasi-covering.
\item[(iii)] If $\gtm^\diam\in\betad(M)$, we have $\Specd(\pi)(\gtm^\diam)\in\betad(N)$. In addition, $\Specd(\pi)|_{\betad(M)}=\betad(\pi)$.
\item[(iv)] If $\gtn^\diam\in\betad(N)$, it holds $\Specd(\pi)^{-1}(\gtn^\diam)$ is a finite subset of $\betad(M)$.
\item[(v)] If $y\in N$ and $\pi^{-1}(y):=\{x_1,\ldots,x_r\}$, we have $\Specd(\pi)^{-1}(\gtn^\diam_y):=\{\gtm^\diam_{x_1},\ldots,\gtm^\diam_{x_r}\}$.
\end{itemize}
\end{prop}
\begin{proof}
(i) Let $f\in{\mathcal S}^\diam(M)$. Let $\sigma_k\in\Z[{\tt x}_1,\dots,{\tt x}_d]$ be the $k$th elementary symmetric form (for $1\leq k\leq d$) and consider the functions $\sigma_k(f):N\to\R,\ y\mapsto \sigma_k(f(x_{1}),\dots,f(x_{d}))$, where $\pi^{-1}(y):=\{x_1,\dots,x_d\}$. By Lemma \ref{sym} each $\sigma_k(f)\in{\mathcal S}^\diam(N)$. As $f$ is a root of the polynomial 
$$
{\tt p}({\tt t}):={\tt t}^d+\sum_{k=1}^d(-1)^k\sigma_k(f){\tt t}^{d-k}\in{\mathcal S}^\diam(N)[{\tt t}],
$$
we conclude $f$ is integral over ${\mathcal S}^\diam(N)$ via $\varphi^\diam_{\pi}$. This means that $\varphi^\diam_{\pi}$ is an integral homomorphism. 

(ii) By Lemma \ref{propern}(ii) the spectral map $\Specd(\pi)$ is open, closed and surjective. In addition, $\Specd(\pi)$ has finite fibers by \cite[Prop.11]{s3}. We prove next that $\Specd(\pi)$ is separated.

Given $\gtp_1,\gtp_2\in\Specd(M)$ with $\Specd(\pi)(\gtp_1)=\Specd(\pi)(\gtp_2)$ we have by (i) and \cite[Cor.5.9]{am} $\gtp_1\not\subset\gtp_2$ and $\gtp_2\not\subset\gtp_1$. Thus, by Lemma \ref{sep} $\gtp_1$ and $\gtp_2$ have disjoint open neighborhoods.

(iii) As $\Specd(\pi)$ is closed, it maps closed points to closed points and the statement follows.

(iv) By Lemma \ref{opcl}(i) we have 
$$
\cl_{\Specd(M)}(\Specd(\pi)^{-1}(\gtn^\diam))=\Specd(\pi)^{-1}(\cl_{\Specd(M)}(\gtn^\diam))=\Specd(\pi)^{-1}(\gtn^\diam).
$$
Write $\Specd(\pi)^{-1}(\gtn^\diam):=\{\gtp_1,\ldots,\gtp_r\}$. We have $\{\gtp_1,\ldots,\gtp_r\}=\bigcup_{i=1}^r\cl_{\Specd(M)}(\gtp_i)$. Let $\gtm^\diam_i$ be the unique maximal ideal of ${\mathcal S}^\diam(M)$ that contains $\gtp_i$. Then $\gtm_i^\diam\in\cl_{\Specs(M)}(\gtp_i)\subset\{\gtp_1,\ldots,\gtp_r\}$ for $i=1,\ldots,r$. As $\Specd(\pi)$ is separated, $\gtp_i\not\subset\cl_{\Specd(M)}(\gtp_j)$ if $i\neq j$. We conclude $\gtp_i=\gtm_i^\diam$ for $i=1,\ldots,r$, so $\Specd(\pi)^{-1}(\gtn^\diam)=\{\gtm_1^\diam,\ldots,\gtm_r^\diam\}\subset\betad(M)$.

(v) As $\pi:M\to N$ is a closed map with finite fibers, it is proper, so by Lemma \ref{propern}(ii) $\Specd(\pi)(\partial M)\cap N=\varnothing$. By (iv) $\Specd(\pi)^{-1}(\gtn_y^\diam)\subset\betad(M)$ is a finite set. Thus, $\Specd(\pi)^{-1}(\gtn_y^\diam)=\{\gtm_{x_1}^\diam,\ldots,\gtm_{x_r}^\diam\}$, as required. 
\end{proof}

Next result points out the good properties of the minimal elements of the collapsing set in the ${\mathcal S}$-case.

\begin{cor}[Minimal elements of the collapsing set]\label{minz} 
Let $\pi:M\to N$ be a semialgebraic $d$-branched covering. Then each minimal element $\gtP\in\Cc_{\Specs(\pi)}$ is a $z$-ideal of ${\mathcal S}(M)$.
\end{cor}
\begin{proof} 
If $\gtP$ is a minimal prime ideal of ${\mathcal S}(M)$, the statement follows for instance from \cite[Cor.4.7]{fe1}. Thus, we may assume $\gtP$ is not a minimal prime ideal.

Let $\gtQ_1$ be a minimal prime ideal of ${\mathcal S}(M)$ contained in $\gtP$. By Lemma \ref{mintomin}(ii) its image $\gtq:=\Specs(\pi)(\gtQ_1)\subset\Specs(\pi)(\gtP):=\gtp$ is a minimal prime ideal of ${\mathcal S}(N)$ and $\Specs(\pi)^{-1}(\gtp)=\{\gtP\}$ because $\gtP\in\Cc_{\Specs(\pi)}$. Write $\Specs(\pi)^{-1}(\gtq):=\{\gtQ_1,\dots,\gtQ_\ell\}$ for some $\ell\leq d$. As $\Specs(\pi)$ is separated, the fiber $\Specs(\pi)^{-1}(\gtq)$ has the trivial topology. We claim: \em $\gtQ_j$ is a minimal prime ideal of ${\mathcal S}(M)$ for $1\leq j\leq\ell$\em. 

Assume $\gtQ_j$ is not a minimal prime ideal of ${\mathcal S}(M)$ for some $j=2,\ldots,\ell$ and let $\gtP'$ be a prime ideal of ${\mathcal S}(M)$ strictly contained in $\gtQ_j$. Then $\Specs(\pi)(\gtP')\subset\Specs(\pi)(\gtQ_j)=\gtq$. As the latter is a minimal prime ideal of ${\mathcal S}(N)$, we have $\Specs(\pi)(\gtP')=\gtq$, that is, $\gtP'\in\Specs(\pi)^{-1}(\gtq)$. This is a contradiction because the fiber $\Specs(\pi)^{-1}(\gtq)$ does not contain a pair of prime ideals such that $\gtP'\subsetneq\gtQ_j$ (recall that $\Spec(\pi)$ is by Proposition \ref{sep2}(ii) a finite quasi-covering). 

We prove next: \em $\gtQ_j\subset\gtP$ for $1\leq j\leq\ell$\em.

As $\gtq=\Specs(\pi)(\gtQ_j)=\Specs(\pi)(\gtQ_1)\subset\Specs(\pi)(\gtP):=\gtp$ and $\Specs(\pi)$ is a closed continuous map, $\gtp\in\cl_{\Specs(N)}(\{\Specs(\pi)(\gtQ_j)\})=\Specs(\pi)(\cl_{\Specs(M)}(\{\gtQ_j\}))$. As $\Specs(\pi)^{-1}(\gtp)=\{\gtP\}$, this implies $\gtP\in\cl_{\Specs(M)}(\{\gtQ_j\})$, so $\gtQ_j\subset\gtP$. 

By Lemma \ref{min}(ii) the sum $\gtQ_1+\cdots+\gtQ_\ell$ is a prime $z$-ideal contained in $\gtP$. To prove that $\gtP$ is a prime $z$-ideal, it is enough to check: $\gtP=\gtQ_1+\cdots+\gtQ_\ell$. As $\gtP$ is a minimal element of $\Cc_{\Specs(\pi)}$, it suffices to show: $\gtQ_1+\cdots+\gtQ_\ell\in\Cc_{\Specs(\pi)}$. Denote $\gtq':=\Specs(\pi)(\gtQ_1+\cdots+\gtQ_\ell)$ and let us prove: \em $\gtQ_1+\cdots+\gtQ_\ell$ is the unique point in the fiber $\Specs(\pi)^{-1}(\gtq')$\em. 

Pick a point $\gtq'_1\in\Specs(\pi)^{-1}(\gtq')$. As $
\gtq=\Specs(\pi)(\gtQ_1)\subset\Specs(\pi)(\gtQ_1+\cdots+\gtQ_\ell)=\gtq'=\Specs(\pi)(\gtq'_1)$, there exists by Lemma \ref{mintomin}(i) a point in the fiber of $\gtq$ contained in $\gtq'_1$. Thus, $\gtQ_k\subset\gtq'_1$ for some index $1\leq k\leq\ell$. Consequently, $\gtq'_1$ and $\gtQ_1+\cdots+\gtQ_\ell$ are prime ideals of ${\mathcal S}(M)$ containing $\gtQ_k$. As the prime ideals of ${\mathcal S}(M)$ containing $\gtQ_k$ constitute a chain, either $\gtq'_1\subset\gtQ_1+\cdots+\gtQ_\ell$ or $\gtQ_1+\cdots+\gtQ_\ell\subset\gtq'_1$. As the fiber $\Specs(\pi)^{-1}(\gtq')$ has the trivial topology, $\gtq'_1=\gtQ_1+\cdots+\gtQ_\ell$, so $\gtQ_1+\cdots+\gtQ_\ell\in\Cc_{\Specs(\pi)}$, as required.
\end{proof}

\section{Proof of Theorem \ref{colapse}}\label{mu}

To get a better understanding of the finite quasi-covering $\Specs(\pi)$ induced by a semialgebraic branched covering $\pi:M\to N$ we prove Theorem \ref{colapse}, which provides a precise description of the subset $\Cc_{\Specs(\pi)}$. Its proof does not involve Theorem \ref{bc}. We need the following notion. 

\begin{defn}
Let $\pi:M\to N$ be a semialgebraic $d$-branched covering and let $b_{\pi}:M\to\Z$ be the branching index of $\pi$. We define the map
$$
\mu^\diam:{\mathcal S}^\diam(M)\to{\mathcal S}^\diam(N),\ f\mapsto\mu^\diam(f):=\tfrac{1}{d}\sigma_1(f)=\tfrac{1}{d}\sum_{x\in\pi^{-1}(y)}b_{\pi}(x)f(x),
$$
where $\sigma_1(\x_1,\ldots,\x_d):=\x_1+\cdots+\x_d$ is the first elementary symmetric form in $d$ variables. 
\end{defn}

\begin{remarks}\label{module}
(i) If $b_{\pi}(x)=1$ for each point in the fiber of a point $y\in N$, then $\pi^{-1}(y)$ consists of $d$ points, so $\mu^\diam(f)(y)$ is the arithmetic mean of the values of $f$ on the points of $\pi^{-1}(y)$. In general, $\mu^\diam(f)(y)$ is a weighted arithmetic mean of the values of $f$ on $\pi^{-1}(y)$.

(ii) The homomorphism $\varphi_{\pi}$ endows ${\mathcal S}^\diam(M)$ with a natural structure of ${\mathcal S}^\diam(N)$-module and the map $\mu^\diam:{\mathcal S}^\diam(M)\to{\mathcal S}^\diam(N)$ is a homomorphism of ${\mathcal S}^\diam(N)$-modules. 

For each $g\in{\mathcal S}^\diam(N)$ and each $y\in N$ we have
\begin{equation*}
\begin{split}
\big(\mu^\diam(g\circ\pi)\big)(y)&=\tfrac{1}{d}\sum_{x\in\pi^{-1}(y)}b_{\pi}(x)(g\circ\pi)(x)\\
&=\tfrac{1}{d}\sum_{x\in\pi^{-1}(y)}b_{\pi}(x)g(y)=g(y)\Big(\tfrac{1}{d}\sum_{x\in\pi^{-1}(y)}b_{\pi}(x)\Big)=g(y),
\end{split}
\end{equation*}
so $\mu^\diam(g\circ\pi)=g$.\qed
\end{remarks}

We also need the following result.
\begin{lem}\label{homeo:fe_FG} Let $N\subset M\subset \R^n$ be semialgebraic sets and let ${\tt j}:N\hookrightarrow M$ be the inclusion. Then there exists $h\in \mathcal{S}^\diam(M)$ with $Z(h)\subset\cl_M(N)$ nowhere dense in $\cl_M(N)$ such that 
$$
\Spec^\diam(N)\setminus \Specd({\tt j})^{-1}(\Zzd(h))=\Spec^\diam(N)\setminus \Zzd(h|_N)
$$
is homeomorphic to $\cl_{\Spec^\diam(M)}(N)\setminus \Zzd(h)$ via $\Specd({\tt j})$. In addition, if $N$ is locally compact, we may assume $Z(h)=\cl_M(N)\setminus N$. If $N$ is closed in $M$, then $\Spec^\diam(N)$ is homeomorphic to $\cl_{\Spec^\diam(M)}(N)$ via $\Specd({\tt j})$.
\end{lem}
\begin{proof}
Let $H\in \mathcal{S}^*(\R^n)$ be such that $Z(H)=\cl_{\R^n}(\cl_{\R^n}(N)\setminus N)$ and define $h:=H|_{M}\in {\mathcal S}^*(M)$. The difference $\cl_{\R^n}(N)\setminus Z(H)=N\setminus Z(H)$ is a dense subset of $N$ (see \cite[\S 2.2]{fe3}). In particular, $\cl_M(N)\setminus Z(h)=N\setminus Z(H)$ is also a dense subset of $\cl_M(N)$.

Let ${\tt j}_1:N\hookrightarrow\cl_M(N)$ be the inclusion. We claim:\em 
$$
\Specd({\tt j}_1)|:\Specd(N)\setminus\Specd({\tt j}_1)^{-1}(\Zzd(h|_{\cl_{M}(N)}))\to \Specd(\cl_M(N))\setminus\Zzd(h|_{\cl_{M}(N)})
$$
is a homeomorphism\em. 

We provide different proofs of the claim for the $\mathcal{S}$-case and the $\mathcal{S}^*$-case. The claim for the $\mathcal{S}$-case is a straightforward consequence of \cite[Lem.1.1]{fe3}. We approach next the $\mathcal{S}^*$-case. Denote $\Zz:=\cl_{\Speca(\cl_M(N))}(\cl_M(N)\setminus N)$. By \cite[Thm.1.2]{fe3} the restriction
$$
\Speca(N)\setminus\Speca({\tt j}_1)^{-1}(\Zz)\to\Speca(\cl_M(N))\setminus\Zz,
$$
of $\Speca({\tt j}_1)$ to $\Speca(N)\setminus\Speca({\tt j}_1)^{-1}(\Zz)$ is a homeomorphism. As $\Zz\subset \Zz^*(h|_{\cl_M(N)})$, also the restriction 
$$
\Speca({\tt j}_1)|:\Speca(N)\setminus \Speca({\tt j}_1)^{-1}(\Zz^*(h|_{\cl_M(N)}))\rightarrow \Speca(\cl_M(N))\setminus\Zz^*(h|_{\cl_M(N)})
$$ 
is a homeomorphism. Note that $\Speca({\tt j}_1)^{-1}(\Zz^*(h|_{\cl_{M}(N)}))=\Zz^*(h|_N)$. 

Next, let ${\tt j}_2:\cl_M(N)\hookrightarrow M$ be the inclusion. As $\cl_M(N)$ is closed in $M$, $\Speca(\cl_M(N))$ is by \cite[Cor.4.6]{fg2} homeomorphic to $\cl_{\Speca(M)}(\cl_M(N))=\cl_{\Speca(M)}(N)$ via $\Speca({\tt j}_2)$. Thus,
$$
\Speca({\tt j}_2)|:\Speca(\cl_M(N))\setminus\Zz^*(h|_{\cl_M(N)})\to\cl_{\Speca(M)}(N)\setminus\Zz^*(h)
$$
is a homeomorphism and $\Speca({\tt j})|=\Speca({\tt j}_2)|\circ \Speca({\tt j}_1)|$ is a homeomorphism too. 

If $N$ is locally compact, $\cl_{\mathbb{R}^n}(N)\setminus N$ is a closed subset of $\R^n$. Thus, $Z(h)=Z(H)\cap M=(\cl_{\mathbb{R}^n}(N)\setminus N)\cap M=\cl_M(N)\setminus N$. Finally, if $N$ is closed in $M$, then $\cl_M(N)=N$ and $\Spec^\diam(N)$ is by \cite[Cor.4.6]{fg2} homeomorphic to $\cl_{\Spec^\diam(M)}(N)$ via $\Specd({\tt j})$, as required.
\end{proof}

We are ready to prove Theorem \ref{colapse}.

\begin{proof}[Proof of Theorem \em \ref{colapse}]
Consider the commutative diagram
$$
\xymatrix{
{\mathcal S}^*(N)\ar[r]^{\varphi^*_\pi}\ar@{^{(}->}[d]&{\mathcal S}^*(M)\ar@{^{(}->}[d]\\
{\mathcal S}(N)\ar[r]^{\varphi_\pi}&{\mathcal S}(M)
}
$$
As $\pi$ is surjective, $f\in{\mathcal S}(N)$ is bounded if and only if $\varphi_\pi(f)=f\circ\pi\in{\mathcal S}(M)$ is bounded. If $\gtq\in\Specs(M)$, then
$$
\Speca(\pi)(\gtq\cap{\mathcal S}^*(M))=(\varphi_\pi^*)^{-1}(\gtq\cap{\mathcal S}^*(M))=\varphi_\pi^{-1}(\gtq)\cap{\mathcal S}^*(N)=\Specs(\pi)(\gtq)\cap{\mathcal S}^*(N).
$$
Along the proof we will make use of Remark \ref{r:maximal} without mention. Statements (iii) and (iv) follow from the equality $\Cc_{\betad(\pi)}=\Cc_{\Specd(\pi)}\cap\betad(M)$ (that follows from Proposition \ref{sep2}(iii) and (iv)). So, let us prove (i) and (ii). Inside the proof of (i) we make a reduction (see \ref{mec} and \ref{mec2}) of the ${\mathcal S}^*$-case to the ${\mathcal S}$-case and we take advantage of it also in the proof of (ii). 

(i) Define ${\mathcal T}^\diam:=\{\gtp\in\Specd(M):\ \ker(\mu^\diam)\subset\gtp\}$, which is a closed subset of $\Specd(M)$. 

We prove first: ${\mathcal T}^\diam\subset\Cc_{\Specd(\pi)}$. Suppose there exists $\gtp\in{\mathcal T}^\diam\setminus\Cc_{\Specd(\pi)}$. Then there exists $\gtp_1\in\Specd(M)\setminus\{\gtp\}$ such that $\Specd(\pi)(\gtp)=\Specd(\pi)(\gtp_1)$. As $\Specd(\pi)$ is a separated map, we achieve a contradiction if we show: $\gtp\subset\gtp_1$.

Pick $f\in\gtp$. If we prove $f^2\in\gtp_1$, then $f\in\gtp_1$, so we assume $f$ is non-negative. By Remark \ref{module}(ii) we have $\mu^\diam(\mu^\diam(f)\circ\pi)=\mu^\diam(f)$. Thus, $f-(\mu^\diam(f)\circ\pi)\in\ker(\mu^\diam)\subset\gtp$, so 
$$
\mu^\diam(f)\circ\pi=f-(f-(\mu^\diam(f)\circ\pi))\in\gtp.
$$ 
Consequently, $\mu^\diam(f)\in\Specd(\pi)(\gtp)=\Specd(\pi)(\gtp_1)$ and $\mu^\diam(f)\circ\pi\in\gtp_1$. In addition, for each $x\in M$ we have $0\leq f(x)\leq d\cdot(\mu^\diam(f)\circ\pi)(x)$ and $d\cdot(\mu^\diam(f)\circ\pi)\in\gtp_1$. By Lemma \ref{suma}(i) we conclude $f\in\gtp_1$.

Next, we prove the converse inclusion: $\Cc_{\Specd(\pi)}\subset{\mathcal T}^\diam$. Pick $\gtp\in\Cc_{\Specd(\pi)}$ and let $\gtp_0\in\Cc_{\Specd(\pi)}$ be a minimal element of $\Cc_{\Specd(\pi)}$ contained in $\gtp$. It is enough to show: $\ker(\mu^\diam)\subset\gtp_0$. As we have announced above, we make a reduction from the ${\mathcal S}^*$-case to the ${\mathcal S}$-case (see claims \ref{mec} and \ref{mec2} below). Once this is done, we prove the statement in paragraph \ref{mec3}.

Assume $\gtp\in\Cc_{\Speca(\pi)}$. Let $\gtm^*\in\beta^*(M)$ be the unique maximal ideal of ${\mathcal S}^*(M)$ that contains $\gtp$ and let $\gtm\in \beta(M)$ be the unique maximal ideal of ${\mathcal S}(M)$ such that $\gtm\cap{\mathcal S}^*(M)\subset\gtm^*$.

\paragraph{}\label{mec} We claim: $\gtp_0\subset\gtm\cap{\mathcal S}^*(M)$.

Suppose $\gtm\cap{\mathcal S}^*(M)\subsetneq\gtp_0$, so $\gtm\cap{\mathcal S}^*(M)\not\in\Cc_{\Speca(\pi)}$. Thus, there exists a prime ideal $\gtq'\in\Speca(M)$ such that $\gtq'\neq\gtm\cap{\mathcal S}^*(M)$ and $\Speca(\pi)(\gtq')=\Speca(\pi)(\gtm\cap{\mathcal S}^*(M))$. As $\Speca(\pi)$ is a separated map, $\gtq'\not\subset\gtm\cap{\mathcal S}^*(M)$ and $\gtm\cap{\mathcal S}^*(M)\not\subset\gtq'$. 

Let $\gtm_1^*\in\beta^*(M)$ be the unique maximal ideal of ${\mathcal S}^*(M)$ that contains $\gtq'$ and let $\gtm_1\in\beta(M)$ be the unique maximal ideal of ${\mathcal S}(M)$ such that $\gtm_1\cap{\mathcal S}^*(M)\subset\gtm_1^*$. Let us show: $\gtq'=\gtm_1\cap \mathcal{S}^*(M)$.

Recall that either $\gtm_1\cap{\mathcal S}^*(M)\subset\gtq'$ or $\gtq'\subset\gtm_1\cap{\mathcal S}^*(M)$. If $\gtm_1\cap{\mathcal S}^*(M)\subset\gtq'$, then
\begin{multline*}
\Specs(\pi)(\gtm_1)\cap{\mathcal S}^*(N)=\Speca(\pi)(\gtm_1\cap{\mathcal S}^*(M))\\
\subset\Speca(\pi)(\gtq')=\Speca(\pi)(\gtm\cap{\mathcal S}^*(M))=\Specs(\pi)(\gtm)\cap{\mathcal S}^*(N).
\end{multline*}
As $\Specs(\pi)$ is a closed map, $\gtn_1:=\Specs(\pi)(\gtm_1)$ and $\gtn:=\Specs(\pi)(\gtm)$ are maximal ideals of ${\mathcal S}(N)$. Let $\gtn^*,\gtn^*_1\in\beta^*(M)$ be the unique maximal ideals of ${\mathcal S}^*(N)$ that contains $\gtn\cap{\mathcal S}^*(N)$ and $\gtn_1\cap{\mathcal S}^*(N)$. Thus, $\gtn_1\cap{\mathcal S}^*(N)\subset\gtn\cap{\mathcal S}^*(N)\subset\gtn^*$. We conclude $\gtn_1^*=\gtn^*$, so $\gtn_1=\gtn$ and $\Speca(\pi)(\gtm_1\cap{\mathcal S}^*(M))=\Speca(\pi)(\gtq')$. As $\Speca(\pi)$ is separated, $\gtm_1\cap{\mathcal S}^*(M)=\gtq'$.

Otherwise, $\gtq'\subset\gtm_1\cap{\mathcal S}^*(M)$. Then, there exists a prime ideal $\gtq:={\mathcal W}_M^{-1}\gtq'\subset\gtm_1$ of ${\mathcal S}(M)$ such that $\gtq'=\gtq\cap{\mathcal S}^*(M)$ and
\begin{align*}
&\Specs(\pi)(\gtq)\cap{\mathcal S}^*(N)=\Speca(\pi)(\gtq')=\Speca(\pi)(\gtm\cap{\mathcal S}^*(M))=\Specs(\pi)(\gtm)\cap{\mathcal S}^*(N),\\
&\Specs(\pi)(\gtq)\cap{\mathcal S}^*(N)=\Speca(\pi)(\gtq')\subset\Speca(\pi)(\gtm_1\cap{\mathcal S}^*(M))=\Specs(\pi)(\gtm_1)\cap{\mathcal S}^*(N).
\end{align*}
Consequently,
$$
\Specs(\pi)(\gtq)\cap{\mathcal S}^*(N)=\Specs(\pi)(\gtm)\cap{\mathcal S}^*(N)\subset\Specs(\pi)(\gtm_1)\cap{\mathcal S}^*(N). 
$$
As in the previous case, $\Specs(\pi)(\gtm)=\Specs(\pi)(\gtm_1)$ and we deduce $\gtq'=\gtm_1\cap{\mathcal S}^*(M)$. 

Next, as $\gtq'\not\subset\gtm\cap{\mathcal S}^*(M)$ and $\gtm\cap{\mathcal S}^*(M)\not\subset\gtq'$, we have $\gtm_1\neq\gtm$, so $\gtm_1^*\neq\gtm^*$. In addition,
$$
\Speca(\pi)(\gtm_1\cap{\mathcal S}^*(M))=\Speca(\pi)(\gtm\cap{\mathcal S}^*(M))\subset\Speca(\pi)(\gtm_1^*)\cap\Speca(\pi)(\gtm^*).
$$
As $\Speca(\pi)$ is a closed map and ${\mathcal S}^*(N)$ is Gelfand, we conclude $\Speca(\pi)(\gtm_1^*)=\Speca(\pi)(\gtm^*)$, so $\gtm^*\not\in\Cc_{\Speca(\pi)}$. By Lemma \ref{disting} and since $\gtm^*\in\cl_{\Speca(M)}(\gtp_0)$,
$$
\#(\Speca(\pi)^{-1}(\Speca(\pi)(\gtp_0)))\geq\#(\Speca(\pi)^{-1}(\Speca(\pi)(\gtm^*)))\geq2, 
$$
which is a contradiction because $\gtp_0\in\Cc_{\Speca(\pi)}$.

\paragraph{}\label{mec2}
As $\gtp_0\subset\gtm\cap{\mathcal S}^*(M)$, there exists a unique prime ideal $\gtq_0:={\mathcal W}_M^{-1}\gtp_0\subset\gtm$ such that $\gtp_0=\gtq_0\cap{\mathcal S}^*(M)$. We claim: \em $\gtq_0\in\Cc_{\Specs(\pi)}$\em.

Pick $\gtq_1\in\Specs(M)$ such that $\Specs(\pi)(\gtq_0)=\Specs(\pi)(\gtq_1)$. Thus,
\begin{multline*}
\Speca(\pi)(\gtq_0\cap{\mathcal S}^*(M))=\Specs(\pi)(\gtq_0)\cap{\mathcal S}^*(N)\\
=\Specs(\pi)(\gtq_1)\cap{\mathcal S}^*(N)=\Speca(\pi)(\gtq_1\cap{\mathcal S}^*(M)).
\end{multline*}
As $\gtp_0\in\Cc_{\Speca(\pi)}$, it follows $\gtq_1\cap{\mathcal S}^*(M)=\gtq_0\cap{\mathcal S}^*(M)=\gtp_0$, so $\gtq_0=\gtq_1$, as claimed.

\paragraph{}\label{mec3}
By \ref{mec} and \ref{mec2} it is enough to consider the ${\mathcal S}$-case and prove: $\gtp_0\in{\mathcal T}$. Pick $f\in\ker(\mu)$ and let us show: $f\in\gtp_0$. 

Consider the non-negative functions $h_1:=|f|-f$ and $h_2:=|f|+f$. As $h_1h_2=0\in\gtp_0$, we may assume $h_1\in\gtp_0$. As $\pi$ is open, closed and surjective,
$$
g_1:N\to\R,\ y\mapsto\sup\{h_1(x):\ x\in\pi^{-1}(y)\}
$$
is by \cite[Const.3.1]{fg3} a semialgebraic function. By \cite[Eq.$(\ast)$ in Proof Thm.1.5]{fg3} it holds ${\Specs(\pi)}(\Dd(h_1))=\Dd(g_1)$. As $\gtp_0\notin\Dd(h_1)$ and $\{\gtp_0\}={\Specs(\pi)}^{-1}({\Specs(\pi)}(\gtp_0))$, we deduce 
$$
{\Specs(\pi)}(\gtp_0)\notin{\Specs(\pi)}(\Dd(h_1))=\Dd(g_1),
$$ 
so $g_1\circ\pi\in\gtp_0$. By Corollary \ref{minz} $\gtp_0$ is a $z$-ideal, so to prove that $f\in\gtp_0$ it is enough to show: $Z(g_1\circ\pi)\subset Z(f)$. 

Suppose there exists a point $x\in Z(g_1\circ\pi)$ such that $f(x)\neq0$. As $g_1(\pi(x))=0$, the semialgebraic function $h_1$ vanishes identically on the fiber $\pi^{-1}(\pi(x))$ (recall that $h_1\geq0$ on $M$). Thus, $f(z)\geq0$ for each $z\in\pi^{-1}(\pi(x))$ and $f(x)>0$. Hence,
$$
\mu(f)(\pi(x))=\tfrac{1}{d}\sum_{z\in \pi^{-1}(\pi(x))}b_{\pi}(z)f(z)>0,
$$
which is a contradiction because $f\in\ker(\mu)$.

\paragraph{} For the sake of completeness, we prove: \em The prime ideal $\gtq_0\in\Cc_{\Specs(\pi)}$ introduced in \em \ref{mec2} \em is minimal in $\Cc_{\Specs(\pi)}$\em. 

Suppose there exists $\gtq_1\subset\gtq_0$ such that $\gtq_1\in\Cc_{\Specs(\pi)}$. By (i) we have $\ker(\mu)\subset\gtq_1$. Therefore, $\ker(\mu^*)=\ker(\mu)\cap \mathcal{S}^*(M)\subset \gtq_1\cap\mathcal{S}^*(M)$, so $\gtq_1\cap\mathcal{S}^*(M)\in \Cc_{\Speca(\pi)}$. As $\gtq_1\cap\mathcal{S}^*(M)\subset\gtq_0\cap\mathcal{S}^*(M)$ and $\gtq_0\cap\mathcal{S}^*(M)\in \Cc_{\Specs(\pi)}$ is minimal, it follows $\gtq_1\cap\mathcal{S}^*(M)= \gtq_0\cap\mathcal{S}^*(M)$, so $\gtq_1=\gtq_0$.

(ii) We have to prove: $\Cc_{\Specd(\pi)}=\cl_{\Specd(M)}(\Cc_\pi)$. Recall that for each $x_1\in M$ we have $\Specd(\pi)^{-1}(\Specd(\pi)(\gtm_{x_1}))=\{\gtm_{x_1},\ldots,\gtm_{x_r}\}$ where $\pi^{-1}(\pi(x_1))=\{x_1,\ldots,x_r\}$. Thus, $\Cc_\pi=\Cc_{\Specd(\pi)}\cap M$, so $\cl_{\Specd(M)}(\Cc_\pi)\subset\Cc_{\Specd(\pi)}$. Let us prove next the converse inclusion. Pick $\gtp\in\Cc_{\Specd(\pi)}$ and let $\gtp_0\in\Cc_{\Specd(\pi)}$ be a minimal element of $\Cc_{\Specd(\pi)}$ contained in $\gtp$. If we prove that $\gtp_0\in\cl_{\Specd(M)}(\Cc_\pi)$, then $\gtp\in\cl_{\Specd(M)}(\{\gtp_0\})\subset\cl_{\Specd(M)}(\Cc_\pi)$. By \ref{mec} and \ref{mec2} it is enough to consider the ${\mathcal S}$-case. By Corollary \ref{minz} the prime ideal $\gtp_0$ is a $z$-ideal and by \cite[Lem.4.1]{fg2} $\gtP_0:=\Specs(\pi)(\gtp_0)$ is also a $z$-ideal. 

Let $f\in\gtP_0$ be such that ${\tt d}:=\dim(Z(f))=\min\{\dim(Z(g)):\ g\in\gtP_0\}$ and denote $Z:=Z(f)$. As $\gtP_0$ is a $z$-ideal, $\gtP_0\in\cl_{\Specs(N)}(Z)$ (use \cite[Lem.4.3]{fg2}). Write $T:=\pi^{-1}(Z)$ and consider the restriction map $\pi|_T:T\to Z$. By Hardt's trivialization theorem \cite[9.3.2]{bcr} there exist: 
\begin{itemize}
\item a semialgebraic partition $\{A_1,\dots,A_r\}$ of $Z$, 
\item semialgebraic sets $P_1,\dots,P_r\subset\R^p$ and 
\item semialgebraic homeomorphisms $\theta_{\ell}:A_{\ell}\times P_{\ell}\to \pi^{-1}(A_{\ell})$
\end{itemize}
such that for $1\leq\ell\leq r$ we have the following commutative diagram
$$
\xymatrix{
A_{\ell}\times P_{\ell}\ar@{->}[r]^{\quad\theta_{\ell}\quad}\ar@{->}[rd]^{\pi_{\ell}\quad}&\pi^{-1}(A_{\ell})\ar@{->}[d]^{\pi|_{\pi^{-1}(A_\ell)}\quad}\\
&A_{\ell}}
$$
where $\pi_{\ell}:A_{\ell}\times P_{\ell}\to A_{\ell}$ is the projection onto $A_{\ell}$. Taking a semialgebraic triangulation of $Z$ compatible with $A_1,\ldots,A_r$ we may assume that each $A_i$ is locally compact. As $\pi$ has finite fibers, each $P_\ell$ is a finite set.

As $\cl_{\Specs(N)}(Z)=\bigcup_{\ell=1}^r\cl_{\Specs(N)}(A_\ell)$, we may assume $\gtP_0\in\cl_{\Specs(N)}(A_1)$. By \cite[Lem.4.3]{fg2} it follows that for each $g\in\mathcal{S}(N)$ such that $Z(g)=\cl_N(A_1)$ we have $g\in\gtP_0$, so $\dim(A_1)=\dim(\cl_N(A_1))\geq {\tt d}=\dim(Z)\geq\dim(A_1)$, that is, $\dim(A_1)={\tt d}$. As $A_1$ is locally compact, the semialgebraic set $C:=\cl_N(A_1)\setminus A_1$ is closed in $N$.

By Lemma \ref{homeo:fe_FG} there exists $h\in {\mathcal S}(N)$ with $Z(h)=C$ such that the inclusion ${\tt j}:A_1\to N$ induces a homeomorphism
$$
\Specs({\tt j})|:\Specs(A_1)\setminus \Zz(h|_{A_1})\to \cl_{\Specs(N)}(A_1)\setminus \Zz(h).
$$
As $\gtP_0$ is a $z$-ideal and $\dim(Z(h))=\dim(C)<{\tt d}$, we have $\gtP_0\not\in\Zz(h)$. In particular, there exists $\gtP_0'\in\Specs(A_1)\setminus \Zz(h|_{A_1})$ such that $\Specs({\tt j})(\gtP_0')=\gtP_0$.

Next, consider the ${\tt d}$-dimensional subset $\pi^{-1}(A_1)$ of $M$ (recall that $\pi$ has finite fibers). As $\pi:M\to N$ is closed and has finite fibers, it is a proper map, so $\pi^{-1}(A_1)$ is locally compact. Thus, $C':=\cl_M(\pi^{-1}(A_1))\setminus \pi^{-1}(A_1)$ is a closed subset of $M$. By Lemma \ref{opcl} we have $C'=\pi^{-1}(C)$ and $\cl_{\Specs(M)}(C')=\Specs(\pi)^{-1}(\cl_{\Specs(N)}(C))$. As $\gtP_0\not\in\Zz(h)$, we deduce $\gtP_0\not\in\cl_{\Specs(N)}(C)$ and $\gtp_0\notin \cl_{\Specs(M)}(C')$.

As before, there exists $h'\in {\mathcal S}(M)$ with $Z(h')=C'$ such that the inclusion ${\tt i}:\pi^{-1}(A_1)\to M$ induces a homeomorphism
$$
\Specs({\tt i})|:\Specs(\pi^{-1}(A_1))\setminus \Zz(h'|_{\pi^{-1}(A_1)})\to \cl_{\Specs(N)}(\pi^{-1}(A_1))\setminus \Zz(h').
$$
As $\gtp_0$ is $z$-ideal and $\gtp_0\notin \cl_{\Specs(M)}(C')$, it follows from \cite[Lem.4.3]{fg2} that $\gtp_0\notin \Zz(h')$. Thus, there exists $\gtp_0'\in\Specs(\pi^{-1}(A_1))\setminus \Zz(h'|_{\pi^{-1}(A_1)})$ such that $\Specs({\tt i})(\gtp_0')=\gtp_0$.

Suppose that $\gtp_0\not\in\cl_{\Specs(M)}(\Cc_\pi)$. Observe that $A_1\cap\pi(\Cc_\pi)\neq\varnothing$ if and only if $\#(P_1)=1$. If such is the case, then $A_1\subset\pi(\Cc_\pi)$. Thus,
\begin{multline*}
\gtP_0=\Specs(\pi)(\gtp_0)\in\cl_{\Specs(N)}(A_1)\subset\cl_{\Specs(N)}(\pi(\Cc_\pi))\\
=\cl_{\Specs(N)}(\Specs(\pi)(\Cc_\pi))=\Specs(\pi)(\cl_{\Specs(M)}(\Cc_\pi)).
\end{multline*}
As $\{\gtp_0\}=\Specs(\pi)^{-1}(\gtP_0)$, we deduce $\gtp_0\in\cl_{\Specs(M)}(\Cc_\pi)$, against our assumption. Consequently, $\#(P_1)\geq2$ and $A_1\cap\pi(\Cc_\pi)=\varnothing$. 

Consider the commutative diagram
$$
\xymatrix{
\Specs(A_1)\times P_1\ar@{->}[rr]^{\quad\Specs(\theta_1)\quad}\ar@{->}[rrd]_{\Specs(\pi_1)\quad}&&\Specs(\pi^{-1}(A_1))\ar@{->}[d]^{\Specs(\pi|_{\pi^{-1}(A_1)})\quad} \ar@{->}[rr]^{\quad\Specs({\tt i})\quad}&&\Specs(M)\ar@{->}[d]^{\Specs(\pi)}\\
&&\Specs(A_1) \ar@{->}[rr]^{\quad\Specs({\tt j})\quad}&& \Specs(N)
}
$$
where $\Specs(\theta_1)$ is a homeomorphism. Thus, there exists $\gtp'_1\in\Specs(\pi^{-1}(A_1))$ such that $\gtp'_1\neq \gtp'_0$ and $\Specs(\pi|_{\pi^{-1}(A_1)})(\gtp_0')=\Specs(\pi|_{\pi^{-1}(A_1)})(\gtp_1')=\gtP_0'$. Observe that $\gtp'_1\notin \Zz(h'|_{\pi^{-1}(A_1)})$. Define 
$$
\gtp_1:=\Specs({\tt i})(\gtp_1')\in\cl_{\Specs(M)}(\pi^{-1}(A_1))\setminus \Zz(h').
$$
As $\gtp_1'\neq\gtp_0'$ and $\Specs({\tt i})|$ is bijective, we deduce $\gtp_1\neq\gtp_0$. As $\Specs(\pi)(\gtp_1)=\gtP_0=\Specs(\pi)(\gtp_0)$ we have $\gtp_0\not\in\Cc_{\Specs(\pi)}$, which is a contradiction. Consequently, $\gtp_0\in\cl_{\Specs(M)}(\Cc_\pi)$, as required.
\end{proof}

\section{Proof of Theorem \ref{bc}}\label{s6}

In this section we prove Theorem \ref{bc}. 

\begin{proof}[Proof of Theorem \em \ref{bc}] 
The implication (ii) $\Longrightarrow$ (iii) follows from Proposition \ref{sep2}(iii) and (iv) and the density of $\betad(M)$ in $\Specs(M)$, whereas the implication (iii) $\Longrightarrow$ (i) follows from Proposition \ref{sep2}(v) and the density of $M$ in $\betad(M)$. The equalities ${\mathcal B}_{\betad(\pi)}=\cl_{\betad(M)}({\mathcal B}_\pi)$ and ${\mathcal R}_{\betad(\pi)}=\cl_{\betad(N)}({\mathcal R}_\pi)$ follow from the equalities ${\mathcal B}_{\Specd(\pi)}=\cl_{\Specd(M)}({\mathcal B}_\pi)$ and ${\mathcal R}_{\Specd(\pi)}=\cl_{\Specd(N)}({\mathcal R}_\pi)$ (once they are proved) together with Proposition \ref{sep2}(iv).

(i) $\Longrightarrow$ (ii). Let $N_1,\ldots,N_r$ be the connected components of $N$ and denote $M_i:=\pi^{-1}(N_i)$. By Lemma \ref{rcc} there exist integers $d_i\geq1$ such that $\pi|_{M_i}:M_i\to N_i$ is a $d_i$-branched covering. In addition, by \cite[Cor.4.7]{fg2} $\Specd(N_1),\ldots,\Specd(N_r)$ are the connected components of $\Specd(N)$. By \cite[Cor.4.6]{fg2} and Corollary \ref{rcc} it is enough to prove Theorem \ref{bc} for the $d_i$-branched coverings $\pi|_{M_i}:M_i\to N_i$. We may assume from the beginning that $N$ is connected.

 By Remark \ref{cuenta}(ii) $\pi:M\to N$ is a $d$-branched covering for some integer $d\geq1$. By Proposition \ref{sep2}(ii) $\Specd(\pi):\Specd(M)\to\Specd(N)$ is a finite quasi-covering. 

\paragraph{} Let us show now: \em the fibers of $\Specd(\pi)$ have no more than $d$ points\em. 

Otherwise, there exists $\gtq\in\Specd(N)$ such that $\#(\Specd(\pi)^{-1}(\gtq))>d$. As $\Specd(\pi)$ is separated, there exists by Lemma \ref{disting} an open neighborhood $\mathcal{V}$ of $\gtq$ in $\Specd(N)$ such that $\#(\Specd(\pi)^{-1}(\gtp))>d$ for each $\gtp\in\mathcal{V}$. As $N$ is dense in $\Specd(N)$, there exists $\gtn_y\in\mathcal{V}\cap N$. Write $\pi^{-1}(y):=\{x_1,\ldots,x_r\}$. By Proposition \ref{sep2}(v) $\Specd(\pi)^{-1}(\gtn_y)=\{\gtm_{x_1},\ldots,\gtm_{x_r}\}$, so $d<r$, which is a contradiction because $\pi$ is a $d$-branched covering. 

\paragraph{} We claim: $\Specd(N)\setminus{\mathcal R}_{\Specd(\pi)}=\{\gtq\in\Specd(N):\ \#(\Specd(\pi)^{-1}(\gtq))=d\}$. 

The inclusion right to left follows from Lemma \ref{disting} and Corollary \ref{max}. To prove the converse inclusion let $\gtq\in\Specd(N)\setminus{\mathcal R}_{\Specd(\pi)}$. By Lemma \ref{genbranch} there exists an open neighborhood $\mathcal{W}$ of $\gtq$ such that the cardinality of the fibers at the points of $\mathcal{W}$ is a constant $c$. As $N\setminus \mathcal{R}_\pi$ is dense in $\Specd(N)$, the intersection $\pi^{-1}(\mathcal{W})\cap (N\setminus \mathcal{R}_\pi)$ is non-empty, so $c=d$, as claimed. 

\paragraph{}\label{pa:regular} We prove next: \em $\gtm_x\in{\Specd(M)}_{\reg}$ for $x\in M_{\reg}$\em.

By the previous claim it is enough to check that $\#(\pi^{-1}(\pi(x)))=d$, which is true by Proposition \ref{sep2}(v) because $x\in M_{\reg}$.

\paragraph{}\label{unbr} Consequently, the restriction 
$$
\Specd(\pi)|_{\Specd(M)_{\reg}}:\Specd(M)_{\reg}\to\Specd(N)\setminus{\mathcal R}_{\Specd(\pi)}
$$
is a $d$-unbranched covering. As $M_{\reg}$ is dense in $M$, it follows from \ref{pa:regular} that
$$
\Specd(M)_{\reg}=\Specd(M)\setminus\Specd(\pi)^{-1}({\mathcal R}_{\Specd(\pi)})
$$ 
is dense in $\Specd(M)$.

\paragraph{}\label{neigh1} 
Let $\gtq\in\Specd(N)$ and write $\Specd(\pi)^{-1}(\gtq):=\{\gtp_1,\ldots,\gtp_r\}$. We claim: \em there exist $g\in{\mathcal S}^\diam(N)$ and $f_1,\ldots,f_r\in{\mathcal S}^\diam(M)$ such that $\gtq\in\Ddd(g)$, $\gtp_i\in\Ddd(f_i)$, $f_if_j=0$ if $i\neq j$,
$$
\Specd(\pi)^{-1}(\Ddd(g))=\bigsqcup_{i=1}^r\Ddd(f_i)\quad\text{and}\quad\Specd(\pi)(\Ddd(f_i))=\Ddd(g)
$$
for $i=1,\ldots,r$.
\em

For each pair $1\leq i<j\leq r$ there exist by Lemma \ref{sep} semialgebraic functions $f_{ij},f_{ji}\in{\mathcal S}^\diam(M)$ such that $\gtp_i\in \Ddd(f_{ij})$, $\gtp_j\in \Ddd(f_{ji})$ and $f_{ij}f_{ji}=0$. For each $i=1,\ldots,r$ define $h_i:=\prod_{k,\,k\neq i}f_{ik}\in{\mathcal S}^\diam(M)$. It holds $\gtp_i\in\Ddd(h_i)$ and $h_ih_j=0$ if $i\neq j$. Observe that $\Ddd(h_i)\cap\Ddd(h_j)=\varnothing$ if $i\neq j$. Define
$$
\mathcal{W}:=\Big(\Specd(N)\setminus\Specd(\pi)\Big(\Specd(M)\setminus\bigcup_{i=1}^r\Ddd(h_i)\Big)\Big)\cap \bigcap_{i=1}^r\Specd(\pi)\big(\Ddd(h_i)\big),
$$
which is an open neighborhood of $\gtq$ in $\Specd(N)$ such that $\Specd(\pi)^{-1}(\mathcal{W})\subset\bigcup_{i=1}^r\Ddd(h_i)$. Let $g\in{\mathcal S}^\diam(N)$ be such that $\gtq\in\Ddd(g)\subset \mathcal{W}$. As $\Specd(\pi)^{-1}(\Ddd(g))=\Ddd(g\circ\pi)$, it follows
$$
\Ddd(h_i)\cap\Specd(\pi)^{-1}(\Ddd(g))=\Ddd(h_i(g\circ\pi)).
$$
If we define $f_i:=h_i(g\circ\pi)$ for $i=1,\ldots,r$, the reader can check that the claim follows.

\paragraph{}\label{pa:branched} 
We check next: \em$\Specd(\pi):\Specd(M)\to\Specd(N)$ is a $d$-branched covering map\em.

By Remark \ref{intersection} it is enough to show that each $\gtq\in \mathcal{R}_\pi$ has a special neighborhood. Write $\Specd(\pi)^{-1}(\gtq):=\{\gtp_1,\ldots,\gtp_r\}$ where $r<d$. Let $g\in{\mathcal S}^\diam(N)$ and $f_1,\ldots,f_r\in{\mathcal S}^\diam(M)$ be as in \ref{neigh1} for $\gtq$ and $\gtp_1,\ldots,\gtp_r$. Let $E_1,\ldots,E_s$ be the connected components of $D(g)$. By Lemma \ref{ccD} the connected components ${\mathcal V}_i:=\cl_{\Specd(N)}(E_i)\cap\Ddd(g)$ of $\Ddd(g)$ are open subsets of $\Specd(N)$.

We may assume $\gtq\in{\mathcal V}:={\mathcal V}_1$. If ${\mathcal U}:=\Specd(\pi)^{-1}({\mathcal V})$ and ${\mathcal U}_i:={\mathcal U}\cap\Ddd(f_i)$, we have ${\mathcal U}=\bigsqcup_{i=1}^r{\mathcal U}_i$ and $\Specd(\pi)({\mathcal U}_i)={\mathcal V}$. Observe that $\gtp_i\in{\mathcal U}_i$ for $i=1,\ldots,r$. 

By \ref{unbr} and Lemma \ref{restr} the restriction
$$
\Specd(\pi)|_{{\mathcal U}\cap\Specd(M)_{\reg}}:{\mathcal U}\cap\Specd(M)_{\reg}\to{\mathcal V}\setminus{\mathcal R}_{\Specd(\pi)}
$$
is a $d$-unbranched covering. By Lemma \ref{opcl2} the restriction
$$
\Specd(\pi)|_{{\mathcal U}_1\cap\Specd(M)_{\reg}}:{\mathcal U}_1\cap\Specd(M)_{\reg}\to{\mathcal V}\setminus{\mathcal R}_{\Specd(\pi)}
$$
is a branched covering with empty ramification set. We prove below in \ref{bl1} that there exists an open dense subset ${\mathcal G}$ of $\mathcal{U}_1$ such that the cardinality of the fibers of $\Specd(\pi)|_{{\mathcal G}}:\mathcal{G}\to \Specd(\pi)({\mathcal G})$ is a constant $e\in \N$. We deduce by Lemma \ref{cuenta0} that ${\mathcal U}_1$ is an exceptional neighborhood of $\gtp_1$ with respect to ${\mathcal V}$. As this can be done with each ${\mathcal U}_i$, we conclude that ${\mathcal V}$ is a special neighborhood of $\gtq$ and $\Specd(\pi)$ is a $d$-branched covering. 

\paragraph{}\label{bl1} We claim: \em there exists an open dense subset ${\mathcal G}$ of ${\mathcal U}_1$ such that the cardinality of the fibers of the restriction map $\Specd(\pi)|_{\mathcal G}$ is a constant $e\in \N$.\em

Denote $E:=E_1$, $D:=\pi^{-1}(E)$ and $D_i:=D(f_i)\cap D$ for $i=1,\ldots,r$. By \ref{neigh1} and Proposition \ref{sep2}(v) we have $D=\bigsqcup_{i=1}^rD_i$ and $\pi(D_i)=E$ for $i=1,\ldots,r$. By Lemma \ref{restr} the restriction $\pi|_D:D\to E$ is a $d$-branched covering. By Remark \ref{cuenta}(ii) and Lemma \ref{opcl2} $\pi|_{D_1}:D_1\to E$ is an $e$-branched covering for some integer $e\geq1$. By \ref{unbr} applied to $\pi|_{D_1}$,
\begin{equation}\label{d1}
\Specd(\pi|_{D_1}):\Specd(D_1)_{\reg}\to\Specd(E)\setminus{\mathcal R}_{\pi|_{D_1}}
\end{equation}
is an $e$-unbranched covering.
As $E$ is dense in ${\mathcal V}$ and $\Specd(\pi)|_{\mathcal U}:{\mathcal U}\to{\mathcal V}$ is by Lemma \ref{opcl} open, closed and surjective, $D$ is dense in ${\mathcal U}$. As $D\cap{\mathcal U}_1=D_1$, we deduce that $D_1$ is dense in ${\mathcal U}_1
$. 

Let ${\tt i}: D_1\to M$ and ${\tt j}: E \to N$ be the inclusions. Consider the commutative diagrams
$$
{\small\xymatrix{
D_1\ar[d]_{\pi|_{D_1}}\ar@{^{(}->}[r]^(0.5){\tt i}& M\ar@<0ex>[d]^\pi&&\Specd(D_1)\ar[d]_{\Specd(\pi|_{D_1})}\ar[rr]^(0.5){\Specd({\tt i})}&&\Specd(M)\ar@<0ex>[d]^{\Specd(\pi)}\\
E\ar@{^{(}->}[r]^(0.5){\tt j}& N&&\Specd(E)\ar[rr]^(0.5){\Specd({\tt j})}&&\Specd(N)
}}
$$
By Lemma \ref{homeo:fe_FG} there exist $a\in{\mathcal S}^\diam(M)$ and $b\in{\mathcal S}^\diam(N)$ such that 
\begin{align*}
&\Specd({\tt i})|:\Specd(D_1)\setminus\Specd({\tt i})^{-1}(\Zzd(a))\to\cl_{\Specd(M)}(D_1)\setminus\Zzd(a),\\
&\Specd({\tt j})|:\Specd(E)\setminus\Specd({\tt j})^{-1}(\Zzd(b))\to\cl_{\Specd(N)}(E)\setminus\Zzd(b),
\end{align*}
are homeomorphisms,
\begin{itemize}
\item $\cl_M(D_1)\setminus Z(a)$ is dense in $\cl_M(D_1)$ and
\item $\cl_N(E)\setminus Z(b)$ is dense in $\cl_N(E)$.
\end{itemize}
In particular, $\cl_{\Specd(M)}(D_1)\setminus\Zzd(a)$ is dense in $\cl_{\Specd(M)}(D_1)$ and $\cl_{\Specd(N)}(E)\setminus\Zzd(b)$ is dense in $\cl_{\Specd(N)}(E)$. As $D_1$ is dense in ${\mathcal U}_1$ and $E$ is dense in $\mathcal{V}$, we deduce 
\begin{align*}
{\mathcal U}_1&\subset\cl_{\Specd(M)}({\mathcal U}_1)=\cl_{\Specd(M)}(D_1),\\
{\mathcal V}&\subset\cl_{\Specd(M)}({\mathcal V})=\cl_{\Specd(M)}(E).
\end{align*}
Define $\Zz_1:=\mathcal{U}_1\cap \Zz^\diam(a)$ and $\Zz_2:=\mathcal{V}\cap \Zz^\diam(b)$, which are closed nowhere dense subsets of $\mathcal{U}_1$ and $\mathcal{V}$. As $\Specd(\pi)|_{\mathcal{U}_1}:{\mathcal{U}_1}\to {\mathcal{V}}$ is a finite quasi-covering and $\Specd(M)_{\reg}\cap {\mathcal{U}_1}\subset \mathcal{U}_{1,\reg}$ is dense in $\mathcal{U}_1$, we have by Lemma \ref{nowhere} that $\Spec(\pi)(\Zz_1)$ is a closed nowhere dense subset of $\mathcal{V}$. Thus,
$$
{\mathcal G}:=\big(\mathcal{U}_1\cap \Specd(M)_{\reg}\big) \setminus \big(\Specd(\pi)^{-1}(\Specd(\pi)(\Zz_1)\cap \Zz_2)\big)
$$
is an open dense subset of $\mathcal{U}_1$. As the spectral map \eqref{d1} is an $e$-unbranched covering, we deduce (via the homeomorphisms $\Specd({\tt i})|$ and $\Specd({\tt j})|$) that the cardinality of the fibers of the restriction $\Specd(\pi)|_{{\mathcal G}}:{\mathcal G}\to \Specd(\pi)({\mathcal G})\subset \mathcal{V}\setminus{\mathcal R}_{\Specd(\pi)}$ is also $e$, as claimed.

To finish the proof of Theorem \ref{bc} it only remains to check the equalities $\mathcal{B}_{\Specd(\pi)}=\cl_{\Specd(M)}({\mathcal B}_\pi)$ and ${\mathcal R}_{\Specd(\pi)}=\cl_{\Specd(N)}({\mathcal R}_\pi).$

\paragraph{}\label{closureofB}
We prove first: $\mathcal{B}_{\Specd(\pi)}=\cl_{\Specd(M)}({\mathcal B}_\pi)$. 

The inclusion $\mathcal{B}_{\pi}\subset \mathcal{B}_{\Specd(\pi)}$ is clear. As $\mathcal{B}_{\Specd(\pi)}$ is a closed subset of $\Specd(M)$, it holds $\cl_{\Specd(M)}({\mathcal B}_\pi)\subset \mathcal{B}_{\Specd(\pi)}$. To prove the converse, pick $\gtp_1\in \mathcal{B}_{\Specd(\pi)}$ and let us show: $\gtp_1\in\cl_{\Specd(M)}({\mathcal B}_\pi)$. 

Denote $\gtq:=\Specd(\gtp_1)$ and $\Specd(\pi)^{-1}(\gtq)=\{\gtp_1,\ldots,\gtp_r\}$. As $\Specd(\pi):\Specd(M)\to\Specd(N)$ is a $d$-branched covering, there exist a special neighborhood $\mathcal{V}$ of $\gtq$ and corresponding exceptional neighborhoods $\mathcal{U}_1,\ldots,\mathcal{U}_r$ of $\gtp_1,\ldots,\gtp_r$. Let $h_i \in{\mathcal S}^\diam(M)$ be such that $\gtp_i\in\Dd^\diam(h_i)\subset \mathcal{U}_i$ for $i=1,\ldots,r$. Arguing as in the proof of \ref{neigh1}, we obtain functions $g\in{\mathcal S}^\diam(N)$ and $f_1,\ldots,f_r \in{\mathcal S}^\diam(M)$ such that $\Dd^\diam(g)$ is a special neighborhood of $\gtq$ and $\gtp_i\in \Dd^\diam(f_i)$ for $i=1,\ldots,r$ are exceptional neighborhoods with respect to $\Dd^\diam(g)$ (see Remark \ref{cuenta}(iii)). In particular, $\Specd(\pi)|_{\Dd^\diam(f_1)}:\Dd^\diam(f_1)\to\Dd^\diam(g)$ is an $e$-branched covering whose collapsing set contains $\gtp_1$. Note that $b_{\Specd(\pi)}(\gtp_1)=e$ and as $\gtp_1\in{\mathcal B}_{\Specd(\pi)}$, we deduce $e>1$. We also point out: $\cl_{\Specd(M)}(\Dd^\diam(f_1))\cap \Dd^\diam(f_j)=\varnothing$ for each $j\neq 1$ and by Lemma \ref{neigh} $\cl_{\Specd(M)}(\Dd^\diam(f_1))=\cl_{\Specd(M)}(D(f_1))$.

Next, consider the open subset $D(g)$ of $N$ and note that $\pi^{-1}(D(g))=\bigsqcup_{i=1}^r D(f_i)$. Thus, by Lemmas \ref{restr} and \ref{opcl2} $\pi|_{D(f_1)}: D(f_1)\to D(g)$ is a branched covering. By Proposition \ref{sep2} it is an $e$-branched covering.

Let ${\tt i}:D(f_1)\hookrightarrow M$ and ${\tt j}:D(g)\hookrightarrow N$ be the inclusions. Denote $Z:=\cl_M(D(f_1))\setminus D(f_1)$ and note that as $\cl_M(D(f_1))\cap D(f_j)=\varnothing$ for $j\neq 1$, we have by Lemma \ref{opcl} that $\pi(Z)=\cl_N(D(g))\setminus D(g)$. By Lemma \ref{homeo:fe_FG}
$$
\Specd({\tt i})|:\Specd(D(f_1))\setminus\Specd({\tt i})^{-1}(\Zz)\to\cl_{\Specd(M)}(D(f_1))\setminus\Zz
$$
is a homeomorphism, where $\Zz:=\cl_{\Specd(M)}(Z)$. Similarly, 
$$
\Specd({\tt j})|:\Specd(D(g))\setminus\Specd({\tt j})^{-1}(\Zz')\to\cl_{\Specd(N)}(D(g))\setminus\Zz'
$$
is a homeomorphism, where $\Zz':=\cl_{\Specd(N)}(\pi(Z))=\Specd(\pi)(\Zz)$.

We claim: \em $\gtp_1\in \cl_{\Specd(M)}(D(f_1))\setminus \Zz$ and in particular $\gtq=\Specd(\gtp_1)\in \cl_{\Specd(N)}(D(g))\setminus \Zz'$.\em

Suppose $\gtp_1\in \Zz$. As $\gtp_1\in \Dd^\diam(f_1)$, we deduce $\Dd^\diam(f_1)\cap Z\neq \varnothing$, so there exists $x\in \cl_M(D(f_1))\setminus D(f_1)$ such that $\gtm^\diam_x\in \Dd^\diam(f_1)$, which is a contradiction because $\Dd^\diam(f_1)\cap M=D(f_1)$.

We have the following commutative diagrams
$$
\xymatrix{
D(f_1)\ar@{^{(}->}[r]^{{\tt i}}\ar[d]_{\pi|_{D(f_1)}}&M\ar[d]^\pi& \Specd(D(f_1))\setminus \Specd({\tt i})^{-1}(\Zz)\ar[rr]^(0.55){\Specd({\tt i})|}\ar[d]_{\Specd(\pi|_{D(f_1)})}&& \cl_{\Specd(M)}(D(f_1))\setminus \Zz\ar@<-5ex>[d]_{\Specd(\pi)|}\\
D(g)\ar@{^{(}->}[r]^{{\tt j}}&N&\Specd(D(g))\setminus \Specd({\tt j})^{-1}(\Zz')\ar[rr]^(0.55){\Specd({\tt j})|}&&\cl_{\Specd(N)}(D(g))\setminus \Zz'
}
$$
where the maps $\Specd({\tt i})|$ and $\Specd({\tt j})|$ are (as proved above) homeomorphisms. As
$$
\Specd(\pi)^{-1}(\gtq)\cap \cl_{\Specd(M)}(D(f_1))=\{\gtp_1\}
$$
we deduce from Theorem \ref{colapse} that $\Specd({\tt i})^{-1}(\gtp_1)\in \mathcal{C}_{\Specd(\pi|_{D(f_1)})}=\cl_{\Specd(D(f_1))}(\mathcal{C}_{\pi|_{D(f_1)}})$. By Lemma \ref{colapseinB} we conclude $\gtp_1\in \cl_{\Specd(M)}(\mathcal{C}_{\pi|_{D(f_1)}})\subset \cl_{\Specd(M)}(\mathcal{B}_{\pi})$.

\paragraph{} By Lemma \ref{opcl}
\begin{multline*}
{\mathcal R}_{\Specd(\pi)}=\Specd(\pi)({\mathcal B}_{\Specd(\pi)})=\Specd(\pi)(\cl_{\Specd(M)}({\mathcal B}_\pi))\\
=\cl_{\Specd(N)}(\Specd(\pi)({\mathcal B}_\pi))=\cl_{\Specd(N)}(\pi({\mathcal B}_\pi))=\cl_{\Specd(N)}({\mathcal R}_\pi)
\end{multline*}
and in addition
\begin{multline*}
\Specd(\pi)^{-1}({\mathcal R}_{\Specd(\pi)})=\Specd(\pi)^{-1}(\cl_{\Specd(N)}({\mathcal R}_\pi))\\
=\cl_{\Specd(N)}(\Specd(\pi)^{-1}({\mathcal R}_\pi))=\cl_{\Specd(M)}(\pi^{-1}({\mathcal R}_\pi)).
\end{multline*}
This means that $\Specd(M)_{\reg}=\Specd(M)\setminus\cl_{\Specd(M)}(\pi^{-1}({\mathcal R}_\pi))$, as required.
\end{proof}

\begin{remarks}[Ramification index of the spectral map]\label{ri}
Let $\pi:M\to N$ be a semialgebraic $d$-bran\-ched covering and let $\Specd(\pi):\Specd(M)\to\Specd(N)$ be the associated spectral map, which is by Theorem \ref{bc} a $d$-branched covering. 

(i) Fix an integer $e\geq2$ and let us check: $\{b_{\Specd(\pi)}\geq e\}=\cl_{\Specd(M)}(\{b_\pi\geq e\})$. The latter shows the neat behavior of $b_{\Specd(\pi)}$ with respect to $b_\pi$, because
\begin{multline*}
\{b_{\Specd(\pi)}=e\}=\{b_{\Specd(\pi)}\geq e\}\setminus\{b_{\Specd(\pi)}\geq e+1\}\\
=\cl_{\Specd(M)}(\{b_\pi\geq e\})\setminus\cl_{\Specd(M)}(\{b_\pi\geq e+1\}).
\end{multline*}

Let $\gtp\in\cl_{\Specd(M)}(\{b_\pi\geq e\})$ and let ${\mathcal U}$ be an exceptional neighborhood of $\gtp$. Then $\Specd(\pi)|_{\mathcal U}:{\mathcal U}\to{\mathcal V}:=\Specd(\pi)({\mathcal U})$ is an $(b_{\Specd(\pi)}(\gtp))$-branched covering. In particular, there exists $x\in \{b_\pi\geq e\}$ such that $\gtm^\diam_x\in \mathcal{U}$. As $\mathcal{U}\cap M$ is an open neighborhood of $x\in M$, there exists by Remark \ref{cuenta}(iii) an exceptional neighborhood $U$ of $x$ such that $U\subset \mathcal{U}\cap M$. By Proposition \ref{sep2} and Remark \ref{cuenta}(iv) we deduce
\begin{multline*}
e\leq b_\pi(x)=\max\{\#(\pi^{-1}(y)\cap U):y\in \pi(U)\}\\
\leq\max\{\#(\Specd(\pi)^{-1}(\gtq)\cap \mathcal{U}):\gtq\in \Specd(\pi)(\mathcal{U})\}=b_{\Specd(\pi)}(\gtp).
\end{multline*}
To show the converse inclusion, note that in \ref{closureofB} inside the proof of Theorem \ref{bc} we proved $\{b_{\Specd(\pi)}= e\}\subset \cl_{\Specd(M)}(\{b_\pi= e\})$, so $\{b_{\Specd(\pi)}\geq e\}\subset\cl_{\Specd(M)}(\{b_\pi\geq e\})$.

(ii) For each $x\in M$ we have $b_{\pi}(x)=b_{\Specd(\pi)}(\gtm^\diam_x)$.

Indeed, if we denote $e:=b_{\pi}(x)$, then $\gtm^\diam_x\in \cl_{\Specd(M)}(\{b_\pi\geq e\})$. As $\{b_\pi\geq e+1\}$ is by Remark \ref{cuenta}(iv) a closed subset of $M$, we have $\cl_{\Specd(M)}(\{b_\pi\geq e+1\})\cap M=\{b_\pi\geq e+1\}$, so $\gtm_x\in\cl_{\Specd(M)}(\{b_\pi\geq e\})\setminus\cl_{\Specd(M)}(\{b_\pi\geq e+1\})=\{b_{\Specd(\pi)}=e\}$, as required.\qed
\end{remarks}

\appendix
\section{Bezoutian covering}\label{A}

Let ${\mathcal S}_n$ denote the symmetric group in $n$ symbols. For each $\gamma\in{\mathcal S}_n$ consider the semialgebraic homeomorphism
$$
{\widehat\gamma}:\R^n\to\R^n,\ x:=(x_1,\dots,x_n)\mapsto(x_{\gamma(1)}\dots,x_{\gamma(n)}),
$$
and define the following equivalence relation $E$ in $\R^n$:
$$
E:=\bigcup_{\gamma\in{\mathcal S}_n}\{(x,z)\in\R^n\times\R^n:\ z={\widehat\gamma}(x)\},
$$
which is a closed semialgebraic subset of $\R^n\times\R^n$. In addition, $\pi_1:E\to\R^n,\ (x,z)\mapsto x$ is a proper map because $\pi_1^{-1}([-r,r]^n)\subset[-r,r]^n\times[-r,r]^n$ for each real number $r>0$.

According to \cite[Thm. 1.4]{br} there exist a semialgebraic set $N$, a surjective semialgebraic map $f:\R^n\to N$ and a homeomorphism $g:\R^n/E\to N$ such that $f=g\circ\pi$, where $\pi:\R^n\to\R^n/E$ is the natural projection. We claim: \em the semialgebraic set $N$ and the maps $f$ and $g$ admit a very precise description\em. 

Let $\sigma_k\in\Z[{\tt x}_1,\dots,{\tt x}_n]$ be the elementary symmetric forms for $1\leq k\leq n$ and consider the polynomial map $\sigma:\R^n\to\R^n,\ x\mapsto(\sigma_1(x),\dots,\sigma_n(x))$. Then $N:=\sigma(\R^n)$ is a semialgebraic set, the semialgebraic map $(f:=)\sigma:\R^n\to N$ is surjective and \em $(g:=)\ol{\sigma}:\R^n/E\to N,\ [x]\mapsto\sigma(x)$ is a well-defined bijection\em. 

If $[x]=[z]$, there exists $\gamma\in{\mathcal S}_n$ such that $z={\widehat\gamma}(x)$, so $\sigma(x)=\sigma(z)$ because each component $\sigma_k$ of $\sigma$ is a symmetric polynomial. To prove that $\ol{\sigma}$ is injective pick $x,z\in\R^n$ such that $\sigma(x)=\sigma(z)$. We have 
$$
\prod_{i=1}^n(\t-x_i)=\t^n+\sum_{k=1}^{n}(-1)^k\sigma_k(x)\t^{n-k}=\t^n+\sum_{k=1}^{n}(-1)^k\sigma_k(z)\t^{n-k}=\prod_{i=1}^n(\t-z_i).
$$ 
Thus, there exists $\gamma\in{\mathcal S}_n$ such that $z={\widehat\gamma}(x)$, so $[x]=[z]$. 

Then, $\sigma^{-1}(\sigma(z))=\{\gamma(z):\ \gamma\in{\mathcal S}_n\}$ for each $z\in\R^n$. We have the commutative diagram:
$$
\xymatrix{
\R^n\ar[r]^(0.4)\pi\ar[rd]_\sigma&\R^n/E\ar[d]^{\ol{\sigma}}\\
&N
}
$$
Note that $\ol{\sigma}$ is continuous and let us see: \em $\ol{\sigma}$ is a homeomorphism\em. 

For each $u:=(u_1,\ldots,u_n)\in\R^n$ consider the polynomial $f_u({\tt t}):={\tt t}^n+\sum_{k=1}^{n}(-1)^ku_k{\tt t}^{k}$. Denote $\zeta_1(u),\dots,\zeta_n(u)$ the real parts of the (complex) roots of the polynomial $f_u$. Each value $\zeta_i(u)$ is repeated according to the multiplicity of the corresponding root. We index such values in such a way that $\zeta_1(u)\leq\cdots\leq\zeta_n(u)$. By \cite[\S 13.3]{gj} the functions $\zeta_1,\ldots,\zeta_n:\R^n\to\R$ are continuous. As $N$ is exactly the set of points $a\in\R^n$ such that $f_a$ has $n$ real roots, the map 
\begin{equation}\label{sectt}
s:N\to\R^n,\ a\mapsto(\zeta_1(a),\dots,\zeta_n(a))
\end{equation}
is a continuous section of $\sigma$. In particular, $\ol{\sigma}^{-1}=\pi\circ s$ is continuous, so $\ol{\sigma}$ is a homeomorphism.

We prove next: \em $\zeta_1,\dots,\zeta_n$ have semialgebraic graph\em, so $s:N\to\R^n$ is a semialgebraic map. 

Let ${\tt u}:=({\tt u}_1,\dots,{\tt u}_n)$ and ${\tt z}$ be variables, ${\tt i}:=\sqrt{-1}$. Consider the non-zero polynomial 
$$
{\tt P}({\tt u},{\tt z}):={\tt z}^n+\sum_{j=1}^n{\tt u}_j{\tt z}^{n-j}\in\Z[{\tt u},{\tt z}].
$$

If we write ${\tt z}:={\tt x}+{\tt i}{\tt y}$, we have
$$
{\tt P}({\tt u},{\tt z})=({\tt x}+{\tt i}{\tt y})^n+\sum_{j=1}^n{\tt u}_j({\tt x}+{\tt i}{\tt y})^{n-j}={\tt P}_1({\tt u},{\tt x},{\tt y})+{\tt i}{\tt P}_2({\tt u},{\tt x},{\tt y})
$$
for certain non-zero polynomials ${\tt P}_1,{\tt P}_2\in\Z[{\tt u},{\tt x},{\tt y}]$. Let $\zeta_j(u)+{\tt i}\eta_j(u)\in\C$ be the roots of $f_u$ for $u\in\R^n$ (where $1\leq j\leq n$). Then 
$$
{\tt P}_1(u,\zeta_j(u),\eta_j(u))+{\tt i}{\tt P}_2(u,\zeta_j(u),\eta_j(u))={\tt P}(u,\zeta_j(u)+{\tt i}\eta_j(u))=f_u(\zeta_j(u)+{\tt i}\eta_j(u))=0.
$$ 
Consequently,
$$
{\tt P}_1(u,\zeta_j(u),\eta_j(u))=0,\ {\tt P}_2(u,\zeta_j(u),\eta_j(u))=0.
$$ 
Let ${\tt R}({\tt u},{\tt x})\in\Z[{\tt u},{\tt x}]$ be the resultant, with respect to ${\tt y}$, of the polynomials ${\tt P}_1({\tt u},{\tt x},{\tt y})$ and ${\tt P}_2({\tt u},{\tt x},{\tt y})$. For each $u\in\R^n$ the real number $\eta_j(u)$ is a common root of ${\tt P}_1(u,\zeta_j(u),{\tt y})$ and ${\tt P}_2(u,\zeta_j(u),{\tt y})$, so ${\tt R}(u,\zeta_j(u))=0$ for $u\in\R^n$ and $1\leq j\leq n$. Thus, $\zeta_j$ has semialgebraic graph, as claimed.

For each $p\in\R^n$ the cardinality of the fiber $\sigma^{-1}(\sigma(p))$ is less than or equal to $\ord({\mathcal S}_n)=n!$. The equality is achieved if the coordinates of $x$ are pairwise distinct. Let us check: \em $\sigma:\R^n\to N$ is a semialgebraic finite quasi-covering\em. It is enough to show: \em it is an open and closed map\em, or equivalently, \em $\pi:\R^n\to\R^n/E$ is an open and closed map\em. 

Let $A$ be an open (resp. closed) subset of $\R^n$. Then the union
$$
\pi^{-1}(\pi(A))=\bigcup_{\gamma\in{\mathcal S}_n}{\widehat\gamma}^{-1}(A)
$$
is an open (resp. closed) subset of $\R^n$, so $\pi(A)$ is open (resp. closed) in $\R^n/E$. 

As $\sigma^{-1}(\sigma(z))=\{\gamma(z):\ \gamma\in{\mathcal S}_n\}$ for each $z\in\R^n$, the collapsing set of $\sigma$ is 
$$
\Cc_{\sigma}=\{(z,\ldots,z)\in\R^n:\ z\in\R\},
$$
whereas the branching set of $\sigma$ is 
$$
{\mathcal B}_{\sigma}=\bigcup_{1\leq i<j\leq n}\{(x_1,\dots,x_n)\in\R^n:\ x_i=x_j\},
$$ 
which is a finite union of hyperplanes of $\R^n$ (and it is nowhere dense in $\R^n$). 

The inclusion right to left is clear. Suppose conversely that the coordinates of the point $x:=(x_1,\ldots,x_n)\in \mathbb{R}^n$ are pairwise distinct. Let $I_i\subset\R$ be an open interval that contains $x_i$ and satisfies $I_i\cap I_j=\varnothing$ if $i\neq j$. The restriction of $\sigma$ to $\prod_{i=1}^nI_i$ is a homeomorphism onto its image. Thus, $x\not\in{\mathcal B}_{\sigma}$.

Observe that $\sigma^{-1}(\sigma({\mathcal B}_{\sigma}))={\mathcal B}_{\sigma}$, so $\mathbb{R}^n_{\text{reg}}=\mathbb{R}^n\setminus{\mathcal B}_{\sigma}$.

The restriction map $\sigma|_{\mathbb{R}^n_{\text{reg}}}:\mathbb{R}^n_{\text{reg}}\to N\setminus{\mathcal R}_{\sigma}$ is an $(n!)$-unbranched covering. Let us show next: \em $\sigma$ is an $(n!)$-branched covering\em.

For each $\gamma\in {\mathcal S}_n$ consider the semialgebraic section $s_\gamma:=\widehat{\gamma}\circ s:N\to \R^n$ of $\sigma$, where the semialgebraic map $s$ was defined in \eqref{sectt}. Pick $a\in N$ and write 
$$
f_a({\tt t}):={\tt t}^n+\sum_{k=1}^{n}(-1)^ka_k{\tt t}^{k}=({\tt t}-b_1)^{k_1}\cdots ({\tt t}-b_\ell)^{k_{\ell}}
$$ 
(where $k_1+\cdots+k_\ell=n$). The cardinality of $\pi^{-1}(a)$ is $d:=\frac{n!}{k_1!\cdots k_\ell!}$ and $\sigma^{-1}(a)=\{s_\gamma(a):\ \gamma\in{\mathcal S}_n\}$. Write $\pi^{-1}(a):=\{p_1,\ldots,p_d\}$ and let $V$ be a connected open semialgebraic neighborhood of $a$ in $N$ such that there exist pairwise disjoint connected open semialgebraic neighborhoods $U_1,\ldots,U_d$ of $p_1,\ldots,p_d$ satisfying $\sigma(U_i)=V$ and $\sigma^{-1}(V)=\bigsqcup_{i=1}^dU_i$ (use Lemma \ref{disting}). Define ${\mathcal S}_{n,i}:=\{\gamma\in{\mathcal S}_n:\ s_\gamma(a)=p_i\}$. Thus, ${\mathcal S}_{n,i}\cap{\mathcal S}_{n,j}=\varnothing$ for $i\neq j$ and ${\mathcal S}_n=\bigsqcup_{i=1}^d{\mathcal S}_{n,i}$. In addition, if $i\neq j$, there exists $\gamma_{ij}\in {\mathcal S}_n$ such that $\widehat{\gamma}_{ij}(p_i)=p_j$. The map ${\mathcal S}_{n,i}\to{\mathcal S}_{n,j},\ \gamma\mapsto\gamma_{ij}\circ\gamma$ is a bijection. We deduce that the cardinality of each ${\mathcal S}_{n,i}$ equals $r:=k_1!\cdots k_\ell!$. In addition, $U_i=s_\gamma(V)$ for each $\gamma\in {\mathcal S}_{n,i}$ and each $i=1,\ldots,d$. The reader can check that ${\mathcal B}_{\sigma|_{U_i}}={\mathcal B}_\sigma\cap U_i$, ${\mathcal R}_{\sigma|_{U_i}}={\mathcal R}_\sigma\cap V$ and $U_{i,\reg}=U_i\cap \mathbb{R}^n_{\text{reg}}$. The restriction map $\sigma|_{U_{i,\reg}}:U_{i,\reg}=U_i\cap \mathbb{R}^n_{\text{reg}}\to V\setminus{\mathcal R}_{\sigma|_{U_i}}=V\setminus{\mathcal R}_\sigma$ is an unbranched semialgebraic covering of $r$ sheets (the ramification index at each point $p_i$ is equal to $r$). Consequently, $V$ is a special neighborhood of $a$ and $U_1,\ldots,U_d$ are the corresponding exceptional neighborhoods for $p_1,\ldots,p_d$. 

\bibliographystyle{amsalpha}

\end{document}